\documentclass[11pt,a4paper]{article}

\usepackage{a4wide}
\usepackage{graphicx}
\usepackage{amssymb,amsmath,amsthm,mathrsfs,xfrac,mathtools,}
\usepackage{enumerate, enumitem}
\usepackage[title]{appendix}
\usepackage{amsfonts}
\usepackage{hyperref}
\usepackage[utf8]{inputenc}
\usepackage{color}
\usepackage{pgf,tikz}
\usetikzlibrary{arrows}
\usepackage{color}
\usepackage{caption}
\usepackage{subcaption}

\newtheorem{theorem}{Theorem}

\newtheorem{definition}{Definition}
\newtheorem{remark}{Remark}
\newtheorem{lemma}{Lemma}
\newtheorem{proposition}{Proposition}

\allowdisplaybreaks[4]

\newenvironment{proofof}[1]{\noindent{\emph{Proof of #1.}}}
{$\hfill\Box$\vspace{0.1 cm}\\}

\newcommand{\brho}{\boldsymbol{\rho}}
\newcommand{\bsigma}{\boldsymbol{\sigma}}
\newcommand{\R}{\mathbb R}
\newcommand{\Riem}{\mathcal{R}}
\newcommand{\tv}{\hbox{\textrm{TV}}}
\newcommand{\BV}{\hbox{\textrm{BV}}}
\newcommand{\N}{{\mathbb N}}

\newcommand{\Sim}{\mathcal{S}}
\newcommand{\T}{{\mathcal T}}
\newcommand{\Z}{\mathbb{Z}}

\newcommand{\Flux}{\mathcal{F}}

\newcommand{\eps}{\varepsilon}

\newcommand{\del}{\partial}

\newcommand{\be}{\begin{equation}}
\newcommand{\ee}{\end{equation}}

\newcommand{\modulo}[1]{{\left|#1\right|}}
\newcommand{\norma}[1]{{\left\|#1\right\|}}

\renewcommand{\L}[1]{\mathbf{L^#1}}
\newcommand{\Lloc}[1]{\mathbf{L^{#1}_{loc}}}
\newcommand{\C}[1]{\mathbf{C^{#1}}}
\newcommand{\Cc}[1]{\mathbf{C_c^{#1}}}
\newcommand{\W}[1]{\mathbf{W^{#1}}}

\DeclareMathOperator{\sgn}{sgn}

\makeatletter
\let\@fnsymbol\@arabic
\makeatother

\title{A multi-scale multi-lane model for traffic regulation via autonomous vehicles}

\author{%
  Paola Goatin\footnotemark[1]
 \and  
  Benedetto Piccoli\footnotemark[2]
}
\date{\today}

\begin{document}

\maketitle
 
\footnotetext[1]{
  Universit\'e C\^ote d'Azur, Inria, CNRS, LJAD, Sophia Antipolis, France. E-mail:
  \texttt{paola.goatin@inria.fr}}
  
\footnotetext[2]{{Rutgers University - Camden, Camden, New Jersey, USA}. E-mail:
  \texttt{piccoli@camden.rutgers.edu}}  


%
\begin{abstract}
  \vspace{5pt}
\noindent
We propose a new model for multi-lane traffic with moving bottlenecks, e.g., autonomous vehicles (AV). It consists of a system of  balance laws for traffic in each lane, coupled in the source terms for lane changing, and fully coupled to ODEs for the AVs' trajectories.
More precisely, each AV solves a controlled equation depending on the traffic density, while the PDE on the corresponding lane has a flux constraint at the AV's location. We prove existence of entropy weak solutions, and we characterize the limiting behavior for the source term converging to zero (without AVs), corresponding to a scalar conservation law for the total density.
The convergence in the presence of AVs is more delicate and we show that the limit does not satisfy an entropic equation for the total density as in the original coupled ODE-PDE model. Finally, we illustrate our results via numerical simulations.

  \bigskip


\end{abstract}
                                %
%
%

\section{Introduction}\label{sec:intro}

Multi-lane traffic has attracted the attention of many researchers in transportation science, in particular applied mathematicians and engineers. 
Works on macroscopic models include papers on modeling
\cite{KlarWegener1999} and analysis \cite{ColomboCorli2006,HoldenRisebro2019}.
Recently, extensions to road networks were proposed, see \cite{GoatinRossi2020, GoatinRossiNHM2020}. 
Starting from these results, we aim at including the presence of autonomous vehicles, which act as moving bottlenecks to regulate traffic, e.g., dissipate traffic waves \cite{Bayen_book,stern2018dissipation}. Such regulation problems were addressed on a theoretical basis \cite{daini:hal-04366870, ECC2022, DelleMonache2019, FerraraIncremona2022,  liard2021pde, piacentini2018traffic,  Piacentini2021, TALEBPOUR2016143, Cicic2018, CicicJohansson2019}, with machine learning approaches~\cite{Wu2018, Vinitsky2018} and also via real world experiments~\cite{gunter2020commercially,stern2019quantifying,wu2019tracking}.\\
A number of approaches were developed to deal with moving bottlenecks ~\cite{DMG2014, GGLP2019, Lattanzio2011, Lebacque1998, Leclercq2004, RamadanSeibold2017, SimoniClaudel2017}.
Moreover, extensions to multiple bottlenecks~\cite{GDDMF2023} and to second order models~\cite{VGC2017} were proposed. 
In particular we follow the idea of coupling an ODE for the moving bottleneck to a PDE for the bulk traffic as in \cite{DMG2014,DMG2017,GGLP2019}.
The coupling is realized in the ODE right-hand side, which depends on the PDE solution, and by imposing a flux limiter to the PDE  at the location of the bottleneck. 
Convergence and continuous dependence for such models were achieved \cite{liard2019well,LiardPiccoli2021}.

In this work,
we consider a first order macroscopic multi-lane model~\cite{ColomboCorli2006, HoldenRisebro2019} 
of the form
\begin{equation} \label{eq:multi_model}
\begin{cases}
\del_t \rho_j + \del_x F_j(\rho_j) = \dfrac{1}{\tau}\left(S_{j-1} (\rho_{j-1},\rho_j) - S_{j} (\rho_{j},\rho_{j+1})\right),  \\
\rho_j(0,x) = \rho^0_j(x),
\end{cases}
\qquad 
\begin{array}{l}
x\in\R, ~t\geq 0, \\
j=1,\ldots,M.
\end{array}
\end{equation}
Above, $M$ is the number of lanes on the road, $\rho_j:[0,+\infty[\times\R\to [0,R_j]$  is the vehicle density on lane $j$,
$F_{j}(\rho_j)=\rho_j v_j(\rho_j)$ is the flux function
and the average speed $v_j=v_j(\rho_j):[0,R_j]\to [0,V_j]$ is a strictly decreasing function such that $v_j(0)=V_j$ and  $v_j(R_j)=0$, so that $F_j$ is concave.
The source terms $S_j$ account for 
mass exchanges form lane $j$ to lane $j+1$ (obviously, $S_0=S_M=0$), scaled by a relaxation factor $1/\tau\in\R^+$.
Letting $\tau\to 0$, this corresponds to letting to zero the lane-change relaxation parameter. \\ Throughout the paper, we will assume that
\begin{itemize}
    \item[{\bf (S0)}] $S_j:\R_+^2\to\R$, $j=1,\ldots,M-1$, are Lipschitz continuous in both variables, with Lipschitz constant $\Sim$, and 
    $S_j(0,0)=S_j(R_j,R_{j+1})=0$, $\del_1 S_j(u,w)\geq 0$ and $\del_2 S_j(u,w)\leq 0$, for all $u\in [0,R_j]$ and $w\in [0,R_{j+1}]$.
\end{itemize}

Some results will be subject to further hypotheses:

\begin{itemize}
    \item[{\bf (S1)}] If  $S_j(u,w)=0$, then $u=w$ (in particular, by {\bf (S0)} this implies $R_j=R$ for all $j=1,\ldots,M$);
     \item[{\bf (S2)}] There exists $c>0$ such that $\del_1 S_j(u,w)\geq c$ and $\del_2 S_j(u,w)\leq -c$, for all $u\in [0,R_j]$ and $w\in [0,R_{j+1}]$.
\end{itemize}

To account for the presence of an autonomous vehicle (AV) in lane $j$, we model it as a moving bottleneck~\cite{DMG2014, GGLP2019}
reducing to zero the capacity of the lane at the AV position, see also~\cite{LBCG2020}. Let $y_j^\ell:[0,+\infty[\,\to\R$, $\ell=1,\ldots,N_j$, be the trajectory of the $\ell$-th AV travelling on the lane $j$.
The coupling with~\eqref{eq:multi_model} is realized though the following microscopic ODE and constraint:
\begin{subequations}\label{eq:ODEcoupling}
\begin{align}
&\dot y_j^\ell (t) = \min\left\{u_j^\ell (t), v_j(\rho_j(t,y_j^\ell(t)+))\right\}, & t>0, \label{eq:AVtraj} \\
& y_j^\ell (0) = y_{0,j}^\ell,  & \label{eq:AVic} \\
& \rho_j (t,y_j^\ell(t)) \left(v_j  (\rho_j(t,y_j^\ell (t))) - \dot y_j^\ell (t) \right) \leq  0, &  t>0,  \label{eq:AVcoupling}
\end{align}
\end{subequations}
where $u_j^\ell: [0,+\infty[ \, \to [0,V_j]$ is the control law (desired speed) applied to the AV.
Remark that, when the constraint~\eqref{eq:AVcoupling} is active, i.e. $\min\left\{u_j^\ell (t), v_j(\rho_j(t,y_i^\ell(t)+))\right\} = u_j^\ell (t) < v_j(\rho_j(t,y_i^\ell(t)+))$,
a non-classical shock arises and 
the downstream value of the density must be zero: $\rho_j (t,y_j^\ell(t)+) = 0$, while $v_j (\rho_j(t,y_j^\ell (t)-))=  \dot y_j^\ell (t)$.
In particular, $\rho_j(t,y_j^\ell (t)-) = \hat\rho_{u_j^\ell}$, where $\hat\rho_{u_j^\ell}>0$  is defined by the unique root of the equation $v(\rho)=u_j^\ell$.
We refer the reader to Figure~\ref{fig:constraint} for a graphical representation.

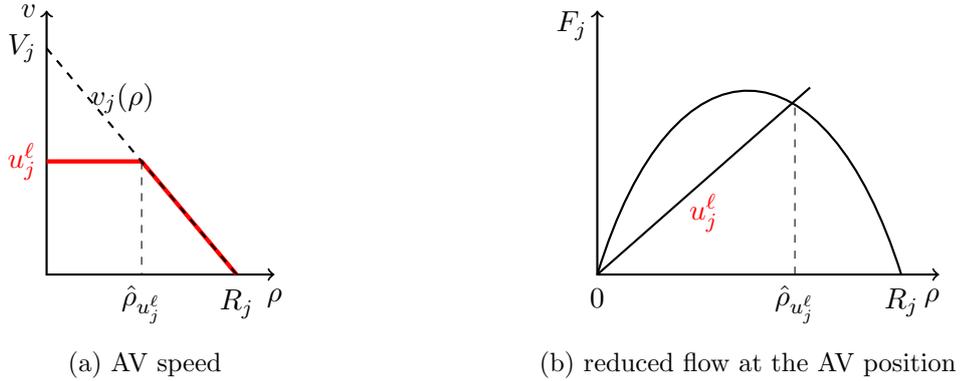
\begin{figure}[ht]
     \centering
     \begin{subfigure}[b]{0.49\textwidth}
         \centering
         \begin{tikzpicture}
\draw [<->,thick] (0,3.5) -- (0,0) -- (3.,0);
\draw [ultra thick,red] (0,1.5) -- (1.25,1.5) -- (2.5,0);
\draw [thick,dashed] (0,3) -- (2.5,0);
\draw [dashed] (1.25,1.5) --(1.25,0);
\node[below] at (3.,-0.05){$\rho$};
\node[below] at (1.25,0){$\hat\rho_{u_j^\ell}$};
\node[below] at (2.5,-0.07){$R_j$};
\node[left] at (0,3.5){$v$};
\node[left] at (0,1.5){$\textcolor{red}{u_j^\ell}$};
\node[left] at (0,3){$V_j$};
\node[above] at (1,2){$v_j(\rho)$};	
\end{tikzpicture}
         \caption{AV speed}
         \label{fig:speed}
     \end{subfigure}
     \hfill
     \begin{subfigure}[b]{0.49\textwidth}
         \centering
         \begin{tikzpicture}
\draw [<->, thick] (0,3.5) -- (0,0) -- (4.5,0);
\draw [thick](0,0) .. controls (1,3.3) and (3,3.2) .. (4,0);
\node [below] at (4.4,-0.05) {$\rho$};
\node [below] at (4,-0.05) {$R_j$};
\node [left] at (0,3.3) {$F_j$};
\node[below] at (0,-0.05){$0$};
\draw [thick](0,0)--(2.8,2.48);
\draw [dashed] (2.6,0)--(2.6,2.3);
\node [below] at (2.6,0) {$\hat{\rho}_{u_j^\ell}$};
\node [below] at (1.4,1.2) {$\textcolor{red}{u_j^\ell}$};
\end{tikzpicture}
         \caption{reduced flow at the AV position}
         \label{fig:flux}
     \end{subfigure}
        \caption{Graphical representation of the AV speed \eqref{eq:AVtraj} and the corresponding flow constraint \eqref{eq:AVcoupling} introducing the corresponding notation.}
        \label{fig:constraint}
\end{figure}


We first provide existence results 
for the coupled system \eqref{eq:multi_model}-\eqref{eq:ODEcoupling},
under assumption {\bf (S0)},
using the wave-front tracking approximation for a fixed $\tau>0$
and assuming bounded variation for the initial datum and the AV control law. Bounds on the total variation are achieved via a careful analysis of the contribution of the source terms and non-classical shocks which may appear at AV locations.\\
Unfortunately, Lipschitz-type estimates cannot be achieved uniformly in $\tau>0$.
Therefore, we first study the limit as $\tau\to 0$ without AVs, which is given by the solution of a scalar conservation law for the sum of the densities over all the lanes.
Such convergence is achieved under assumptions {\bf (S0)} and {\bf (S1)}, adapting a technique originally designed for a chromatography problem, see \cite{BressanShen2000}.\\
We then focus on the case of two lanes and a single AV, 
under assumptions {\bf (S0)}, {\bf (S1)}, and {\bf (S2)}. 
Even in this simplified setting, we show that in the limit  
$\tau\to 0$ we do not recover the entropy weak solutions in the sense of~\cite{DMG2014}, attesting the richness of the model. 

To further illustrate our results, we provide numerical simulations via finite volume schemes.
We start illustrating the dynamics on a 3-lane road with AVs and different speeds. An oscillatory pattern emerges due to the uneven distribution of traffic among lanes.
We then pass to investigate numerically the convergence as $\tau\to 0$ with and without AVs.
In particular, we show the difference between the limit of solutions with AVs and the scalar model proposed in~\cite{DMG2014}
for the sum of densities.

This paper provide further comprehension of the complexities arising when coupling multi-lane traffic with moving bottlenecks, e.g. AVs. Natural next steps include characterizing futher the limit as $\tau\to 0$, understanding continuous dependence, and extending the theory to networks.

\section{Existence of solutions: Wave-front tracking approximations and compactness estimates}\label{sec:existence}

The multi-lane model including moving bottlenecks reads:
\begin{equation} \label{eq:multiLWR+manyAV}
\begin{cases}
\del_t \rho_j + \del_x F_j(\rho_j) = \dfrac{1}{\tau}\left(S_{j-1} (\rho_{j-1},\rho_j) - S_{j} (\rho_{j},\rho_{j+1})\right),  \\
\rho_j(0,x) = \rho^0_j(x),\\
\dot y_j^\ell (t) = \min\{u_j^\ell (t), v_j(\rho_j(t,y_j^\ell(t)+))\}, \\
y_j^\ell (0) = y^\ell_{0,j},   \\
\rho_j (t,y_j^\ell(t)) \left(v_j(\rho_j  (t,y_j^\ell (t))) - \dot y_j^\ell (t) \right) \leq  0,
\end{cases}
\qquad 
\begin{array}{l}
x\in\R, ~t\geq 0, \\
\ell=1,\ldots,N_j, \\
j=1,\ldots,M, 
\end{array}
\end{equation}
Solutions to \eqref{eq:multiLWR+manyAV}
are intended in the following weak sense.

\begin{definition}\label{def:multiWEsol}
  The $(M+\sum_j N_j)$-tuple $\left(\brho^\tau, \bf{y}\right)$, with 
  \[
\brho^\tau=\left(\rho^\tau_1,\ldots,\rho^\tau_M\right)\quad \hbox{and}
  \quad {\bf y}=\left(y^1_1,\ldots,y^{N_1}_1,y^1_2,\ldots, y^{N_M}_M\right), 
  \]
  provides a solution
  to~\eqref{eq:multiLWR+manyAV} with $\brho^0=\left(\rho^0_1,\ldots,\rho^0_M\right)\in\L1 \left(\R; [0,R_1]\times\ldots\times [0,R_M]\right)$
  if the following conditions hold for $\ell=1,\ldots,N_j$, $j=1,\ldots,M$:
  \begin{enumerate}
  \item \label{def:sol-ML:1}
    $\brho^\tau \in \C{0} \left([0,+\infty[; \L1 \left(\R; [0,R_1]\times\ldots\times [0,R_M]\right)\right)$ and
    $\tv\left(\brho^\tau(t)\right) < + \infty$ for all $t>0$;
    

  \item \label{def:sol-ML:2} $y_j^\ell \in \mathbf{W^{1,1}_{loc}}(\R^+;\R)$;

  \item \label{def:sol-ML:3} 
For every $\kappa\in\R$ and for all $\varphi\in\Cc1 (\R^2;\R^+)$  it holds
\begin{align}
    \int_{\R^+}\int_\R &\left(
    |\rho^\tau_j-\kappa| \del_t \varphi +\sgn (\rho^\tau_j-\kappa)(F_j(\rho^\tau_j)-F_j(\kappa)) \del_x\varphi
    \right) dx\, dt 
    +\int_\R |\rho_j^0-\kappa| \varphi(0,x) \, dx \nonumber \\
    & + \dfrac{1}{\tau} \int_{\R^+}\int_\R \sgn (\rho^\tau_j-\kappa) \left(  S_{j-1} (\rho^\tau_{j-1},\rho^\tau_j) - S_{j} (\rho^\tau_{j},\rho^\tau_{j+1})   \right) \varphi \, dx\, dt \nonumber \\
    & {+} \, 2 \sum_{\ell=1}^{N_\ell} \int_{\R^+} \left(
    F_j(\kappa)  - \dot y_j^\ell (t) \kappa \right)^+ \varphi(t,y_j^\ell(t)) \, dt \geq 0\,; \label{eq:multiEC}
\end{align}

  \item \label{def:sol-ML:5}
    For a.e. $t >0$, 
    $F_j\left(\rho^\tau_j\left(t,y_j^\ell(t)\pm\right)\right)  - \dot y_j^\ell(t) \rho^\tau_j \left(t,y_j^\ell(t)\pm\right) \le 0$;
      
        \item \label{def:sol-ML:4}
    For a.e. $t > 0$, $\dot y_j^\ell(t) = 
    \min\left\{u_j^\ell(t), v_j\left(\rho^\tau_j\left(t, y(t)_j^\ell+\right)\right)\right\}$.

  \end{enumerate}
\end{definition}

Existence and uniqueness without moving bottlenecks have been provided in~\cite{GoatinRossi2020, HoldenRisebro2019}.
To include moving bottlenecks, we can use the techniques developed in~\cite{GGLP2019, LBCG2020}.

From now on, we consider for simplicity only one AV, say on lane $i\in\{1,\ldots,M\}$. We denote by $y=y(t)$ its trajectory, and by $u=u(t)$ its desired velocity.
We therefore focus on the well-posedness of the strongly coupled problem
\begin{equation} \label{eq:multiLWR+AV}
\begin{cases}
\del_t \rho_j + \del_x F_j(\rho_j) = \dfrac{1}{\tau}\left(S_{j-1} (\rho_{j-1},\rho_j) - S_{j} (\rho_{j},\rho_{j+1})\right),  \\
\rho_j(0,x) = \rho^0_j(x),\\
\dot y (t) = \min\{u (t), v_i(\rho_i(t,y(t)+))\}, \\
y (0) = y_0,   \\
\rho_i (t,y(t)) \left(v_i(\rho_i  (t,y (t))) - \dot y (t) \right) \leq  0,
\end{cases}
\qquad 
\begin{array}{l}
x\in\R, ~t\geq 0, \\
j=1,\ldots,M. 
\end{array}
\end{equation}
We will prove the following existence result.

\begin{theorem} \label{thm:multiLWR+AV}
Under hypotheses {\bf (S0)}, let the initial conditions $\rho_j^0\in\L1 (\R;[0,R_j])$ for $j=1,\ldots,M$, have finite total variation, and $u\in\L1(\R^+;[0,V_i])$ also have finite total variation.
Then, for any $\tau>0$, there exists a solution $\left(\brho^\tau,y^\tau\right)$ of~\eqref{eq:multiLWR+AV} in the sense of Definition~\ref{def:multiWEsol}.
\end{theorem}

For the proof, we proceed similarly to~\cite{GGLP2019, GoatinLB2019}. \\
We approximate the initial data by a sequence of piece-wise constant functions $\rho_{j}^{0,\nu}$, $j=1,\ldots,M$, such that
each $\rho_{j}^{0,\nu}$ has a finite number of discontinuities and
\[
\lim_{\nu\to +\infty} \norma{\rho_{j}^{0,\nu} - \rho_j^0}_1 =0
\qquad \hbox{and} \qquad
\tv\left(\rho_{j}^{0,\nu}\right) \leq \tv \left(\rho_j^0\right)\quad\hbox{for all }\nu\in\N.
\]
Analogously, we approximate $u$ with piece-wise constant functions $u^\nu$ with a finite number of discontinuities and such that
\[
\lim_{\nu\to +\infty} \norma{u^\nu - u}_1 =0
\qquad \hbox{and} \qquad
\tv\left( u^\nu\right) \leq \tv \left(u\right)\quad\hbox{for all }\nu\in\N.
\]
Consider a sequence of time steps $\Delta t^\nu$, $\nu\in\N$,
such that $\lim_{\nu\to\infty}\Delta t^\nu=0$.
We construct a sequence
$\left(\brho^{\tau,\nu},y_\nu^\tau\right)$ of approximate solutions to~\eqref{eq:multi_model}-\eqref{eq:ODEcoupling} using the Wave Front Tracking (WFT) and operator splitting algorithms as follows.

For any $\nu>0$, we set $t^n:=n\Delta t^\nu$, $n\in\N$.
On the time interval $[t^n,  t^{n+1}[$, $n\in\N$, we define $\left(\brho^{\tau,\nu},y_\nu^\tau\right)$ as the WFT approximation of 
\begin{equation} \label{eq:hom_step}
\begin{cases}
\del_t \rho_j + \del_x F_j(\rho_j) = 0,  \\
\rho_j(t^{n},x) = \rho_{j}^{\tau,\nu}(t^{n},x),\\
\dot y (t) = \min\{u^\nu (t), v_i(\rho_i(t,y(t)+))\}, \\
y (t^{n}) = y_\nu^\tau (t^{n}-),   \\
\rho_i (t,y(t)) \left(v_i(\rho_i  (t,y (t))) - \dot y (t) \right) \leq  0,
\end{cases}
\qquad 
\begin{array}{l}
x\in\R, ~t\in[t^{n}, t^{n+1}[, \\
j=1,\ldots,M, 
\end{array}
\end{equation}
constructed as described in~\cite[Chapter 7]{Bressan_book} or~\cite[Chapter A.7]{Bayen_book}. Note that, due to the presence of the source term, we cannot employ a fixed grid as in~\cite[Chapter 6]{Bressan_book}, see also~\cite[Section 3.1]{GGLP2019}.
At time $t=t^{n+1}$, we set
\begin{equation} \label{eq:update}
\rho_{j}^{\tau,\nu}(t^{n+1},\cdot)=\rho_{j}^{\tau,\nu}(t^{n+1}-,\cdot) 
+ \frac{\Delta t^\nu}{\tau}\left(S_{j-1} (\rho_{j-1}^{\tau,\nu},\rho_{j}^{\tau,\nu}) - S_{j} (\rho_{j}^{\tau,\nu},\rho_{j+1}^{\tau,\nu})\right).
\end{equation}
This sequence of approximate solutions satisfies the following compactness estimates.

\begin{lemma} \label{lem:Linfty}
Let {\bf (S0)} hold and $\Delta t^\nu \leq \tau/2\Sim$. If $\rho_j^0(x)\in [0,R_j]$ for all $j=1,\ldots,M$, then the approximate solutions $\brho^{\tau,\nu}$ satisfy 
\[
\rho_{j}^{\tau,\nu}(t,x) \in [0,R_j] \qquad\hbox{for all }x\in\R,~t>0,~\tau>0.
\]
In particular, the domain $\Pi_{j=1}^M [0,R_j]$ is invariant for~\eqref{eq:multiLWR+AV}.
\end{lemma}

\begin{proof}
Since the homogeneous step~\eqref{eq:hom_step} does not alter the invariant domain, we focus on~\eqref{eq:update}.
We proceed by induction and assume $0\leq \rho_{j}^{\tau,\nu}(t^{n}-,x)\leq R_j$ for all $j=1,\ldots,M$.
Denoting $\rho_{j}^\pm = \rho_{j}^{\tau,\nu}(t^{n}\pm,x)$, we get by hypotheses
\begin{align*}
\rho_{j}^+&=\rho_{j}^- 
+ \frac{\Delta t^\nu}{\tau}\left(S_{j-1} (\rho_{j-1}^-,\rho_{j}^-) - S_{j} (\rho_{j}^-,\rho_{j+1}^-)\right)\\
&=\rho_{j}^- 
+ \frac{\Delta t^\nu}{\tau}\left(S_{j-1} (\rho_{j-1}^-,\rho_{j}^-) - S_{j-1} (0,0) - S_{j} (\rho_{j}^-,\rho_{j+1}^-) + S_{j} (0,0)\right)\\
&=\rho_{j}^- \left[ 1 - \frac{\Delta t^\nu}{\tau}\left( \del_1 S_j -\del_2 S_{j-1} \right)\right]
+ \frac{\Delta t^\nu}{\tau}\left(\del_1S_{j-1} \rho_{j-1}^- -\del_2 S_j  \rho_{j+1}^-\right) \geq 0,
\end{align*} 
where we have used the Mean Value Theorem setting $\del_1S_\ell = \del_1 S_\ell(\xi_{\ell}\rho_{\ell-1}^-,\xi_\ell\rho_{\ell}^-)$ and $\del_2S_\ell = \del_2 S_\ell(\xi_{\ell}\rho_{\ell-1}^-,\xi_\ell\rho_{\ell}^-)$ for some $\xi_\ell\in [0,1]$, $\ell=j-1,j$, and we have exploited {\bf (S)} and the hypothesis $\Delta t^\nu \leq \tau/2\Sim$. 
Similarly, 
\begin{align*}
\rho_{j}^+-R_j&=\rho_{j}^-  -R_j
+ \frac{\Delta t^\nu}{\tau}\left(S_{j-1} (\rho_{j-1}^-,\rho_{j}^-) - S_{j} (\rho_{j}^-,\rho_{j+1}^-)\right)\\
&=\rho_{j}^- -R_j
+ \frac{\Delta t^\nu}{\tau}\left(S_{j-1} (\rho_{j-1}^-,\rho_{j}^-) - S_{j-1} (R_{j-1},R_j) - S_{j} (\rho_{j}^-,\rho_{j+1}^-) + S_{j} (R_j,R_{j+1})\right)\\
&\leq (\rho_{j}^- -R_j) \left[ 1 - \frac{\Delta t^\nu}{\tau}\left( \del_1 S_j -\del_2 S_{j-1} \right)\right]
 \leq 0,
\end{align*} 
with $\del_1S_j = \del_1 S_j(\sigma_{j-1},\sigma_j)$ and $\del_2S_{j-1} = \del_2 S_{j-1}(\sigma_{j-2},\sigma_{j-1})$ for some $\sigma_{\ell-1}=\rho_{\ell-1}^- + \xi_\ell (R_{\ell-1} - \rho_{\ell-1}^-)$, $\sigma_\ell= \rho_{\ell}^- + \xi_\ell (R_{\ell} - \rho_{\ell}^-)$ with $\xi_\ell\in[0,1]$, $\ell=j-1,j$. 
\end{proof}

\begin{lemma} \label{lem:TV}
Let {\bf (S0)} hold and $\Delta t^\nu \leq \tau/2\Sim$, then the approximate solutions $\brho^{\tau,\nu}$ satisfy 
\begin{equation}\label{eq:TVbound}
\tv \left(\rho_{j}^{\tau,\nu}(t,\cdot)\right) \leq  \T,   \qquad\hbox{for all }t>0,~\tau>0,
\end{equation}
for $j\in\{1,\ldots,M\}$, where $\T:=\max_j \tv \left(\rho_j^0\right) +2R_i+C\tv(u)$.
In particular, the sequence total variation in space is uniformly bounded for all $\tau>0$.
\end{lemma}

\begin{proof}
We consider the Glimm type functional
\begin{equation} \label{eq:glimm}
  \Upsilon_i (t) = \Upsilon\left(\rho_{i}^{\tau,\nu}(t,\cdot),u_\nu\right)
  := \tv\left(\rho_{i}^{\tau,\nu}(t,\cdot)\right) + 2R_i + \gamma(t)
  + C \,\tv\left(u^\nu(\cdot); [t,+\infty[\right),
\end{equation}
where $C=-6/\max_{\rho\in[0,R_i]}F_i''(\rho)$ and $\gamma$ is given by 
\begin{equation*}
  \gamma(t):=
  \begin{cases}
    -2\hat\rho_{u^\nu(t)}
    & \hbox{if}~\rho_i^{\tau,\nu}(t,y_\nu^\tau(t)-)=
    \hat\rho_{u^\nu(t)},~\rho_i^{\tau,\nu}(t,y_\nu^\tau(t)+)= 0,
    \\
    0
    &
    \hbox{otherwise},
  \end{cases}
\end{equation*}
we have that $\Upsilon_i (t^{n+1}-) \leq \Upsilon_i (t^{n}+)$. Of course, for $j\not=i$, where the AV is not present, we have 
$\tv \left(\rho_{j}^{\tau,\nu}(t^{n+1}-,\cdot)\right) \leq  \tv \left(\rho_{j}^{\tau,\nu}(t^{n}+,\cdot)\right)$.

It is not restrictive to assume that jump locations in $\rho_{j}^{\tau,\nu}(t^{n+1},\cdot)$ do not coincide for different $j\in\{1,\ldots,M\}$. Denoting by $\rho_{j,L}^\pm$ and $\rho_{j,R}^\pm$ respectively the left and right traces of $\rho_{j}^{\tau,\nu}(t^{n+1}\pm,\cdot)$ at a jump discontinuity, we get from~\eqref{eq:update}
\begin{align*}
    \rho_{j,L}^+ - \rho_{j,R}^+ = \rho_{j,L}^- - \rho_{j,R}^- 
    &+ \frac{\Delta t^\nu}{\tau}\left(S_{j-1} (\rho_{j-1}^{\tau,\nu},\rho_{j,L}^-) - S_{j} (\rho_{j,L}^-,\rho_{j+1}^{\tau,\nu})\right) \\
    &- \frac{\Delta t^\nu}{\tau}\left(S_{j-1} (\rho_{j-1}^{\tau,\nu},\rho_{j,R}^-) - S_{j} (\rho_{j,R}^-,\rho_{j+1}^{\tau,\nu})\right) \\
    = \left(\rho_{j,L}^- - \rho_{j,R}^- \right)
    &\left[1+ \frac{\Delta t^\nu}{\tau}\left( \del_2 S_{j-1} (\rho_{j-1}^{\tau,\nu},\xi_j) - \del_1 S_{j} (\tilde\xi_j,\rho_{j+1}^{\tau,\nu})  \right)\right],
\end{align*}
for some $\xi_j,\tilde\xi_j >0$.
By {\bf (S)}, we have that 
\[
0\leq 1+ \frac{\Delta t^\nu}{\tau}\left( \del_2 S_{j-1} (\rho_{j-1}^{\tau,\nu},\xi_j) - \del_1 S_{j} (\tilde\xi_j,\rho_{j+1}^{\tau,\nu})  \right) \leq 1,
\]
provided that $\Delta t^\nu \leq \tau/2\Sim$. Therefore we conclude that
\[
\tv \left(\rho_{j}^{\tau,\nu}(t^{n+1}+,\cdot)\right) \leq  \tv \left(\rho_{j}^{\tau,\nu}(t^{n+1}-,\cdot)\right) \qquad\hbox{for}~j=1,\ldots,M.
\]
Finally, we get the uniform $\tv$ bounds (independent of $\tau$)
\begin{align}
    &\tv \left(\rho_{j}^{\tau,\nu}(t,\cdot)\right) \leq  \tv \left(\rho_j^0\right), \qquad\hbox{for}~j\in\{1,\ldots,M\}\setminus i, \label{eq:TVwoAV} \\
    &\tv \left(\rho_{i}^{\tau,\nu}(t,\cdot)\right) \leq  \tv \left(\rho_i^0\right) + 2R_i + C\, \tv(u). \label{eq:TVwithAV}
\end{align}
Setting $\T:=\max_j \tv \left(\rho_j^0\right) +2R_i+C\tv(u)$ we get the desired estimate.
\end{proof}

\begin{lemma} \label{lem:timeLip}
Let {\bf (S0)} hold and $\Delta t^\nu \leq \tau/2\Sim$, then  for any $t_1,t_2\in [0,T]$, $T>0$, the approximate solutions $\brho^{\tau,\nu}$ satisfy 
\[
\sum_{j=1}^M\norma{\rho_{j}^{\tau,\nu}(t_1,\cdot) - \rho_{j}^{\tau,\nu}(t_2,\cdot)}_1 \leq (C+C_\tau)\modulo{t_1-t_2} ,   \qquad\hbox{for all }\nu\in\N,
\]
where $C=ML\T$ is independent of $\tau$ (with $L:=\max\left\{\max_j\norma{F_j'}_\infty,\norma{u}_\infty\right\}$) and  $C_\tau=\frac{4\Sim}{\tau}\left( \sum_{j=1}^M \norma{\rho_j^0}_1 + 2ML\T T\right)$.
In particular, the Lipschitz constant does not depend on $\nu$ but blows-up as $\tau\searrow 0$.
\end{lemma}

\begin{proof}
Let us fix $t_1,t_2\in\R$ such that
$0\leq t_1 < t_2$ and suppose that there are $N+1$ time steps between $t_1$ and $t_2$:
\[
t_1\leq t^{k} < t^{k+1} < \ldots < t^{k+N} \leq t_2\qquad\hbox{for some } k\geq 0.
\]
Thus,
\begin{align*}
    &\norma{\rho_{j}^{\tau,\nu}(t_2,\cdot) - \rho_{j}^{\tau,\nu}(t_1,\cdot)}_1
    \leq  
    \norma{\rho_{j}^{\tau,\nu}(t_2,\cdot) - \rho_{j}^{\tau,\nu}(t^{k+N}+,\cdot)}_1
    +
    \norma{\rho_{j}^{\tau,\nu}(t^{k}-,\cdot) - \rho_{j}^{\tau,\nu}(t_1,\cdot)}_1 \\
    & ~~~~+ 
    \sum_{i=k}^{k+N} \norma{\rho_{j}^{\tau,\nu}(t^{i}+,\cdot) - \rho_{j}^{\tau,\nu}(t^{i}-,\cdot)}_1
    + 
    \sum_{i=k}^{k+N-1} \norma{\rho_{j}^{\tau,\nu}(t^{i+1}-,\cdot) - \rho_{j}^{\tau,\nu}(t^{i}+,\cdot)}_1 .
\end{align*}
We observe that, due to finite wave-propagation speed, we have
\[
\norma{\rho_{j}^{\tau,\nu}(t^{i+1}-,\cdot) - \rho_{j}^{\tau,\nu}(t^{i}+,\cdot)}_1 \leq L \T\, \Delta t^\nu,
\]
where $L:=\max\left\{\max_j\norma{F_j'}_\infty,\norma{u}_\infty\right\}$ is the maximum of the Lipschitz constant of $F_j$ and the maximal AV speed. Moreover, by~\eqref{eq:update},
\[
\norma{\rho_{j}^{\tau,\nu}(t^{i}+,\cdot)-\rho_{j}^{\tau,\nu}(t^{i}-,\cdot) }_1
=  \frac{\Delta t^\nu}{\tau}\norma{\left(S_{j-1} (\rho_{j-1}^{\tau,\nu},\rho_{j}^{\tau,\nu}) - S_{j} (\rho_{j}^{\tau,\nu},\rho_{j+1}^{\tau,\nu})\right)(t^{i}-,\cdot)}_1.
\]
Let us define 
\[
g_{j,i}^\pm := \norma{\left(S_{j-1} (\rho_{j-1}^{\tau,\nu},\rho_{j}^{\tau,\nu}) - S_{j} (\rho_{j}^{\tau,\nu},\rho_{j+1}^{\tau,\nu})\right)(t^{i}\pm,\cdot)}_1
\]
and set
\begin{align*}
    \rho_j^\pm &=  \rho_{j,i}^\pm := \rho_{j,\nu}^\tau(t^{i}\pm,\cdot), \\
    S_j^\pm & = S_{j,i}^\pm:= S_{j} (\rho_j^\pm,\rho_{j+1}^\pm), \\
    \del_1 S_j &:= \del_1 S_j \left(\sigma_j,\sigma_{j+1} \right), \quad\hbox{where}~\sigma_j=\rho_{j}^-+ \xi_j(\rho_{j}^+ - \rho_{j}^-), \\
    \del_2 S_j &:= \del_2 S_j \left(\sigma_j,\sigma_{j+1}\right), \quad\hbox{where}~\sigma_{j+1}=\rho_{j+1}^-+ \xi_j(\rho_{j+1}^+ - \rho_{j+1}^-),
\end{align*}
for some $\xi_j\in[0,1]$.

Using the above notations, we develop
\begin{align*}
  &\sum_{j=1}^M  g_{j,i}^+ =  \sum_{j=1}^M \norma{S_{j-1}^+ - S_{j}^+}_1 \\
  &=  \sum_{j=1}^M \norma{S_{j-1}^- - S_{j}^- +\del_1S_{j-1} (\rho_{j-1}^+ - \rho_{j-1}^-) + \left(\del_2S_{j-1}- \del_1S_j\right) (\rho_{j}^+ - \rho_{j}^-) - \del_2S_j (\rho_{j+1}^+ - \rho_{j+1}^-)}_1 \\
  &=  \sum_{j=1}^M \norma{S_{j-1}^- - S_{j}^- + \frac{\Delta t^\nu}{\tau}\left( \del_1S_{j-1} (S_{j-2}^- - S_{j-1}^-) + \left(\del_2S_{j-1}- \del_1S_j\right) (S_{j-1}^- - S_{j}^-) - \del_2S_j (S_{j}^- - S_{j+1}^-)\right) }_1 \\
  &\leq \sum_{j=1}^M \left[\left(1 -\frac{\Delta t^\nu}{\tau} \left(\del_1S_j -\del_2S_{j-1}\right)\right)\norma{S_{j-1}^- - S_{j}^-}_1 + \frac{\Delta t^\nu}{\tau} \del_1S_{j-1} \norma{S_{j-2}^- - S_{j-1}^-}_1
  - \frac{\Delta t^\nu}{\tau} \del_2S_j\norma{S_{j}^- - S_{j+1}^-}_1 \right]\\
  &= \sum_{j=1}^M \left(1 -\frac{\Delta t^\nu}{\tau} \left(\del_1S_j -\del_2S_{j-1}\right)\right)g_{j,i}^-  + \sum_{j=1}^M\frac{\Delta t^\nu}{\tau} \del_1S_{j-1}\, g_{j-1,i}^-
  - \sum_{j=1}^M\frac{\Delta t^\nu}{\tau} \del_2S_j\, g_{j+1,i}^- \\
  &= \sum_{j=1}^M \left(1 -\frac{\Delta t^\nu}{\tau} \left(\del_1S_j -\del_2S_{j-1}\right)\right)g_{j,i}^-  + \sum_{j=0}^{M-1}\frac{\Delta t^\nu}{\tau} \del_1S_{j}\, g_{j,i}^-
  - \sum_{j=2}^{M+1}\frac{\Delta t^\nu}{\tau} \del_2S_{j-1}\, g_{j,i}^- \\
  &= \sum_{j=1}^M  g_{j,i}^-\,,
\end{align*}
where, by abuse of notation, we set $g_{0,i}^-=g_{M+1,i}^-=0$.
Moreover, it holds
\begin{align*}
  &g_{j,i}^- - g_{j,i-1}^+
  = \norma{S_{j-1,i}^- - S_{j,i}^-}_1 - \norma{S_{j-1,i-1}^+ - S_{j,i-1}^+}_1 \\
  &\leq \norma{S_{j-1,i}^- - S_{j,i}^- -S_{j-1,i-1}^+ + S_{j,i-1}^+}_1 \\
  &= \left\|\del_1 S_{j-1}^i \left( \rho_{j-2,i}^- -\rho_{j-2,i-1}^+\right) + 
  \left(\del_2 S_{j-1}^i -\del_1 S_{j}^i \right) \left( \rho_{j-1,i}^- -\rho_{j-1,i-1}^+\right) - 
  \del_2 S_{j}^i \left( \rho_{j,i}^- -\rho_{j,i-1}^+\right)\right\|_1 \\
  &\leq 4\Sim L\T \Delta t^\nu,
\end{align*}
where we have set $\del_l S_{j}^i:=\del_l S_{j}(s_{j-1}^i,s_{j}^i)$, $l=1,2$, with
$s_{j-1}^i= \rho_{j-1,i-1}^+ + \xi_j (\rho_{j-1,i}^- - \rho_{j-1,i-1}^+)$ and $s_{j}^i= \rho_{j,i-1}^+ + \xi_j (\rho_{j,i}^- - \rho_{j,i-1}^+)$ for some $\xi_j\in[0,1]$.\\
Therefore, we get
\begin{align*}
    \sum_{j=1}^M&\norma{\rho_{j}^{\tau,\nu}(t_2,\cdot) - \rho_{j}^{\tau,\nu}(t_1,\cdot)}_1
    \leq M L \T \left( t_2 - t_1\right)
    +  \frac{\Delta t^\nu}{\tau}
    \sum_{i=k}^{k+N} \sum_{j=1}^M g_{j,i}^- \\
    &\leq M L \T \left( t_2 - t_1\right)
    +  \frac{\Delta t^\nu}{\tau}
    \sum_{i=k}^{k+N} \left( \sum_{j=1}^M g_{j,0}^+ + i 4\Sim M L \T \Delta t^\nu \right)\\
    &\leq M L \T \left( t_2 - t_1\right)
    +  \frac{t_2 - t_1}{\tau}
    4\Sim \sum_{j=1}^M \norma{\rho_j^0}_1 +  4\Sim M L \T \Delta t^\nu \frac{\Delta t^\nu}{\tau} \sum_{i=k}^{k+N} i  \\
    &\leq M L \T \left( t_2 - t_1\right)
    +  \frac{t_2 - t_1}{\tau}
    4\Sim \sum_{j=1}^M \norma{\rho_j^0}_1 +  8\Sim M L \T T \frac{t_2 - t_1}{\tau} ,  
\end{align*}
which gives the desired $\L1$-Lipschitz in time estimate uniform in $\nu$ for any $\tau >0$ fixed.
We observe that the estimate blows-up as $\tau\searrow 0$, therefore not allowing to pass to the limit directly in~\eqref{eq:multiLWR+AV}.
\end{proof}

\begin{remark}
    From the above estimates, we note that uniform Lipschitz continuity does not hold even if $\sum_j g^+_{j,0}=0$ (i.e. the initial data are at equilibrium).
\end{remark}

\begin{proofof}{Theorem~\ref{thm:multiLWR+AV}}
For any $T>0$, the previous lemmas provide the compactness estimates on $\left\{\brho^{\tau,\nu}\right\}_\nu$ that ensure the existence of a subsequence, still denoted by $\left\{\brho^{\tau,\nu}\right\}_\nu$, converging in $\L1$ to some function $\brho^\tau \in \C{0} \left([0,T]; \L1 \left(\R; [0,R_1]\times\ldots\times [0,R_M]\right)\right)$ such that
    $\tv\left(\brho^\tau(t)\right) < + \infty$ for a.e. $t\in [0,T]$ (Helly's Theorem), so that point~\ref{def:sol-ML:1}
    in Definition~\ref{def:multiWEsol} is satisfied. 
    
   Concerning $\left\{ y_\nu^\tau\right\}_\nu$, by construction we have that $\dot y_\nu^\tau (t)\in [0,V_i]$ for a.e. $t\in [0,T]$ and $\nu\in\N$. Hence, by the Ascoli-Arzela Theorem, there exists a subsequence, still denoted by $\left\{ y_\nu^\tau\right\}_\nu$ converging uniformly to a function $y^\tau\in\C0 ([0,T];\R)$, which is also Lipschitz of constant $V_i$.
    Following~\cite[Proof of Lemma 3.4]{GGLP2019}, one can prove also that $\dot y_\nu^\tau$ has uniformly bounded total variation on $[0,T]$. 
    Here, in addition to jumps due to interactions with waves coming from the right and jumps in the control function $u^\nu$, we have jumps induced by the relaxation step~\eqref{eq:update} at any $t^n=n \Delta t^\nu$. In this case, $\dot y_\nu^\tau$ can be increasing only if $\rho_i^{\tau,\nu}(\cdot,y_\nu^\tau(\cdot)+)$ decreases at $t=t^n$.
    This increment is bounded by 
    $\dfrac{\Delta t^\nu}{\tau} \norma{v_i'}_\infty \left(\norma{S_{i-1}}_\infty + \norma{S_{i}}_\infty\right)$, thus
    $\dfrac{T}{\tau} \norma{v_i'}_\infty \left(\norma{S_{i-1}}_\infty + \norma{S_{i}}_\infty\right)$ in total
    (note that this bound is not uniform in $\tau$).
    Therefore, $\dot y_\nu^\tau$ converges to $\dot y^\tau$ in $\Lloc{1}([0,T];\R)$ and $y^\tau\in \mathbf{W^{1,1}_{loc}}(\R^+;\R)$, see Definition~\ref{def:multiWEsol}, point~\ref{def:sol-ML:2}.

    Proceeding as in~\cite[Proof of Theorem 3.1]{GGLP2019},  we get that the $(\brho^\tau,y^\tau)$ satisfies
    points~\ref{def:sol-ML:3} and~\ref{def:sol-ML:5}
    in Definition~\ref{def:multiWEsol}.
    
We are now left with Definition~\ref{def:multiWEsol}, point~\ref{def:sol-ML:4}. \
To this end, we follow~\cite[Section 3.3]{LiardPiccoli2021} and~\cite[Section 3]{GGLP2019},
but note that, in our case, $\hat{\rho}_u = \rho_u^*$, where $ \rho_u^*$ is implicitly defined by $v_i( \rho_u^*)=u$, and  we have the additional source term in~\eqref{eq:update}.\\
Fix $\bar{t}\in [0,T]$ and, possibly discarding a set of zero measure, without loss
of generality assume:
\begin{enumerate}
    \item $\lim_{\nu\to\infty}u^\nu(\bar{t})=u(\bar{t})$, $u(\bar{t}-)=u(\bar{t}+)=\bar{u}$;
    \item $y^\tau$ continuously differentiable at $\bar{t}$ and $\lim_{\nu\to\infty}\dot{y}_\nu^\tau(\bar{t})=\dot{y}^\tau(\bar{t})$;
    \item $\dot{y}^{\tau,\nu}(\bar{t})=\min\left\{ u^\nu(\bar{t}),v_i(\rho_i^{\tau,\nu}(\bar{t},y_\nu^\tau(\bar{t})+))\right\}$;
    \item $\lim_{\nu\to\infty}\rho_i^{\tau,\nu}(\bar{t},x)=\rho_i^\tau(\bar{t},x)$ for a.e. $x\in\R$, and $t\not= n\Delta t^\nu$ for every $\nu$, $n\in\N$.
\end{enumerate}
From 4. we have that step ~\eqref{eq:update} does not occur at $\bar{t}$ for any $\nu$, thus $\rho_i^{\tau,\nu}$ is constructed only by wave-front tracking on a sufficiently small open interval containing $\bar{t}$. Therefore, we can follow the same strategy of \cite[Section 3.3]{LiardPiccoli2021} and~\cite[Section 3]{GGLP2019}.\\
By 2. and 3., point~\ref{def:sol-ML:4} of Definition~\ref{def:multiWEsol} is satisfied if 
\begin{equation}\label{eq:to-be-proved}
        \lim_{\nu\to\infty} \min\left\{ u^\nu(t),v_i(\rho_i^{\tau,\nu}(t,y_\nu^\tau(t)+))\right\} =         \min\left\{ u(t),v_i(\rho_i^\tau(t,y^\tau(t)+))\right\}.
\end{equation}
Since by 1. $\lim_{\nu\to\infty}u^\nu(\bar{t})=u(\bar{t})=\bar{u}$, 
\eqref{eq:to-be-proved} holds true if
$v_i(\rho_i^{\tau}(\bar{t},y^\tau(\bar{t})+))\geq \bar{u}$
and $v_i(\rho_i^{\tau,\nu}(\bar{t},y_\nu^\tau(\bar{t})+))\geq \bar{u}$
for infinitely many $\nu$. 
From 4., if 
$v_i(\rho_i^{\tau,\nu}(\bar{t},y_\nu^\tau(\bar{t})+))<\bar{u}$
for $\nu$ sufficiently large, then we also have
$v_i(\rho_i^{\tau}(\bar{t},y^\tau(\bar{t})+))<\bar{u}$.
Define $\rho_\pm=\lim_{x\to y^\tau(\bar{t})\pm} \rho_i^\tau(\bar{t},x)$, then we can restrict to the case $v_i(\rho_i^{\tau}(\bar{t},y^\tau(\bar{t})+))=v_i(\rho_+)<\bar{u}$,  which implies $\rho_+>\rho^*_{\bar{u}}$.\\
We distinguish two cases: 
\begin{itemize}
    \item[{\bf 1.}] $\rho_->\rho^*_{\bar{u}}$:
    First, notice that Lemma 1 and Lemma 4  of \cite{LiardPiccoli2021} hold true since they are based only on $\tv$ bounds that still hold. For some $\epsilon>0$ sufficiently small,
we have that $\rho_--2\epsilon>\rho^*_{\bar{u}}=\hat{\rho}_{\bar{u}}$.
Then the proof of Lemma 6 of \cite{LiardPiccoli2021} still holds,
since it is based on Lemma 4 and the fact that $\rho_--2\epsilon>\hat{\rho}_{\bar{u}}$. We conclude in the same way as in 
\cite[Section 3.3.1]{LiardPiccoli2021}.

\item[{\bf 2.}] $\rho_-\leq \rho^*_{\bar{u}}$: In this case, we have $\rho_-\leq \rho^*_{\bar{u}}<\rho_+$.
Following \cite[Section 3.3.1]{LiardPiccoli2021}, if
$y^\tau(\bar{t})\leq y_\nu^\tau(\bar{t})$ for an infinite number of indices $\nu$, then we conclude by Lemma 4 that \eqref{eq:to-be-proved} holds true. \\
Assume now that $y^\tau(\bar{t})> y_\nu^\tau(\bar{t})$
for $\nu$ sufficiently big. Fix $\epsilon>0$, then
again by Lemma 4 of \cite{LiardPiccoli2021}, there exists $\delta>0$ such that the following holds true. The function $x\to \rho_i^{\tau,\nu}(\bar{t},x)$ takes values in $[\rho_--\epsilon,\rho_-+\epsilon]$ on the interval $[y_\nu^\tau(\bar{t})-\delta, y_\nu^\tau(\bar{t})]$, and takes values in $[\rho_+-\epsilon,\rho_++\epsilon]$ on the interval
$ [y^\tau(\bar{t}),y^\tau(\bar{t})+\delta]$.
Since $\rho_i^{\tau,\nu}$ is generated by wave-front tracking, 
on the interval $I=[y_\nu^\tau(\bar{t}),y^\tau(\bar{t})]$ it may contain only shocks, non-classical shocks, and rarefaction shocks.
The only admissible non-classical shock is $({\rho}^*_{\bar{u}},0)$.
If $\rho_i^{\tau,\nu}$ contains a non-classical shock for infinitely many $\nu$, then the non-classical shock must be located at $y^{\tau,\nu}(\bar{t})$ and by 3.  $\rho_i^{\tau,\nu}(y^{\tau,\nu}(\bar{t})+)=0$.\\
Now, by Lemma 1 of \cite{LiardPiccoli2021}, for $\nu$ large enough, any rarefaction shock contained in $I$ has strength less than $\epsilon$, thus we deduce that $\rho_i^{\tau,\nu}$ takes values in $[\rho_--2\epsilon,\rho_++2\epsilon]$ on the interval $I$, except possibly to the right of the non-classical shock located at $y^{\tau,\nu}(\bar{t})$.
In this case, a classical big shock
is present to the right of $y^{\tau,\nu}(\bar{t})$ connecting $0$ to a value $\tilde{\rho}\in
[\rho_--2\epsilon,\rho_++2\epsilon]$.
This big shock is followed by small waves, classical shocks or rarefaction shocks, with left and right-hand states in
$[\rho_--2\epsilon,\rho_++2\epsilon]$.
Since these small waves have speed strictly lower than the non-classical shock and the big classical shock, they interact with both of them within time
time $\delta/(\bar{u}-\bar{\lambda})$,
where $\bar{u}$ is the speed of the non-classical shock and 
$\bar{\lambda}$ is the maximum
speed of the small waves, thus
$\bar{\lambda}\leq F_j'(\rho_--2\epsilon)$.
In particular, 
we can follow the proof
of Lemmas 13 and 14 of \cite{LiardPiccoli2021},
thus also Lemma 10 of \cite{LiardPiccoli2021} holds true.
Since $\epsilon$ is arbitrarily small, we can conclude as in \cite[Section 3.3.3]{LiardPiccoli2021}.

\end{itemize}
This concludes the proof of Theorem~\ref{thm:multiLWR+AV}.
\end{proofof}


\section{Asymptotic behaviour without AVs}
\label{sec:asymM}

In this section we assume that no AV is present, thus we focus on the multi-lane model~\cite{HoldenRisebro2019}: 
\begin{equation} \label{eq:multiLWR}
\begin{cases}
\del_t \rho_j + \del_x F_j(\rho_j) = \dfrac{1}{\tau}\left(S_{j-1} (\rho_{j-1},\rho_j) - S_{j} (\rho_{j},\rho_{j+1})\right),  \\
\rho_j(0,x) = \rho^0_j(x),
\end{cases}
\qquad 
\begin{array}{l}
x\in\R, ~t\geq 0, \\
j=1,\ldots,M. 
\end{array}
\end{equation}
In this case, the definition of entropy weak solution reduces to:
\begin{definition}\label{def:multiLWREC}
A function $\brho^\tau=\left(\rho^\tau_1,\ldots,\rho^\tau_M\right) \in \C{0} \left([0,+\infty[; \L1 \left(\R; [0,R_1]\times\ldots\times [0,R_M]\right)\right)$ is a entropy weak solution of~\eqref{eq:multiLWR} with $\brho^0\in\L1 \left(\R; [0,R_1]\times\ldots\times [0,R_M]\right)$ if 
for every $\kappa\in\R$ and for all $\varphi\in\Cc1 (\R^2;\R^+)$  it holds
\begin{align}
    \int_{\R^+}\int_\R &\left(
    |\rho^\tau_j-\kappa| \del_t \varphi +\sgn (\rho^\tau_j-\kappa)(F_j(\rho^\tau_j)-F_j(\kappa)) \del_x\varphi
    \right) dx\, dt 
    +\int_\R |\rho_j^0-\kappa| \varphi(0,x) \, dx \nonumber \\
    & + \dfrac{1}{\tau} \int_{\R^+}\int_\R \sgn (\rho^\tau_j-\kappa) \left(  S_{j-1} (\rho^\tau_{j-1},\rho^\tau_j) - S_{j} (\rho^\tau_{j},\rho^\tau_{j+1})   \right) \varphi \, dx\, dt  \geq 0 \label{eq:multiLWREC}
\end{align}
for $j=1,\ldots,M$.
\end{definition}

We extend the $\L1$ stability result given in~\cite[Theorem 3.3]{HoldenRisebro2019} taking in account flux dependency.
Besides~\eqref{eq:multiLWR}, we consider 
\begin{equation} \label{eq:multiLWRbis}
\begin{cases}
\del_t \sigma_j + \del_x \tilde F_j(\sigma_j) = \dfrac{1}{\tau}\left(S_{j-1} (\sigma_{j-1},\sigma_j) - S_{j} (\sigma_{j},\sigma_{j+1})\right),  \\
\sigma_j(0,x) = \sigma^0_j(x),
\end{cases}
\qquad 
\begin{array}{l}
x\in\R, ~t\geq 0, \\
j=1,\ldots,M,
\end{array}
\end{equation}
with $\tilde F$ satisfying the same hypotheses as $F$ ($\tilde F_j(0)=\tilde F_j(R_j)=0$, Lipschitz and concave).

\begin{lemma} \label{lem:multiLWRstability}
Under hypothesis {\bf (S0)}, let $\brho^\tau$ and $\bsigma^\tau$ be respectively the entropy weak solutions of~\eqref{eq:multiLWR} and~\eqref{eq:multiLWRbis}, and
let $\mathfrak{T}:=\min\left\{\tv(\brho^0),\tv(\bsigma^0)\right\}$, $\mathfrak{L}:=\max_{j=1,\ldots,M}\norma{F_j'-\tilde F_j'}_\infty$. Then, for a.e. $t >0$ it holds
\begin{equation} \label{eq:L1contraction}
    \sum_{j=1}^M \norma{\rho_j^\tau(t,\cdot) - \sigma_j^\tau(t,\cdot)}_1
    \leq
    \sum_{j=1}^M \norma{\rho_j^0 - \sigma_j^0}_1
    + M\, \mathfrak{T}\, \mathfrak{L} \, t.
\end{equation}
In particular, the above estimate is independent of $\tau$.
\end{lemma}


\begin{proof}
Let us consider WFT approximate solutions constructed by fractional steps as in Section~\ref{sec:existence}.
Given a sequence of time steps $\Delta t^\nu$, $\nu\in\N$,
such that $\Delta t^\nu \leq \tau/2\Sim$, let $\brho^{\tau,\nu}$ and $\bsigma^{\tau,\nu}$ the sequences of WFT approximations to~\eqref{eq:multiLWR} and~\eqref{eq:multiLWRbis} such that, 
setting  $t^n:=n\Delta t^\nu$, $n\in\N$ for any $\nu>0$ fixed:
\begin{itemize}
    \item 
In any time interval $[t^n,  t^{n+1}[$, $n\in\N$, $\brho^{\tau,\nu}$ is the WFT approximation of 
\begin{equation} \label{eq:hom_step1}
\begin{cases}
\del_t \rho_j + \del_x F_j(\rho_j) = 0,  \\
\rho_j(t^{n},x) = \rho_{j}^{\tau,\nu}(t^{n}+,x),
\end{cases}
\qquad 
\begin{array}{l}
x\in\R, ~t\in[t^{n}, t^{n+1}[, \\
j=1,\ldots,M, 
\end{array}
\end{equation}
 and $\bsigma^{\tau,\nu}$ is the WFT approximation of 
\begin{equation} \label{eq:hom_step2}
\begin{cases}
\del_t \sigma_j + \del_x \tilde F_j(\sigma_j) = 0,  \\
\sigma_j(t^{n},x) = \sigma_{j}^{\tau,\nu}(t^{n}+,x),
\end{cases}
\qquad 
\begin{array}{l}
x\in\R, ~t\in[t^{n}, t^{n+1}[, \\
j=1,\ldots,M, 
\end{array}
\end{equation}
constructed as described e.g. in~\cite[Section 2.3]{HoldenRisebroBook2015}.

\item At time $t=t^{n+1}$, we define
\begin{align} \label{eq:update1}
\rho_{j}^{\tau,\nu}(t^{n+1}+,\cdot)&=\rho_{j}^{\tau,\nu}(t^{n+1}-,\cdot) 
+ \frac{\Delta t^\nu}{\tau}\left(S_{j-1} (\rho_{j-1}^{\tau,\nu},\rho_{j}^{\tau,\nu}) - S_{j} (\rho_{j}^{\tau,\nu},\rho_{j+1}^{\tau,\nu})\right)(t^{n+1}-,\cdot), \\
\label{eq:update2}
\sigma_{j}^{\tau,\nu}(t^{n+1}+,\cdot)&=\sigma_{j}^{\tau,\nu}(t^{n+1}-,\cdot) 
+ \frac{\Delta t^\nu}{\tau}\left(S_{j-1} (\sigma_{j-1}^{\tau,\nu},\sigma_{j}^{\tau,\nu}) - S_{j} (\sigma_{j}^{\tau,\nu},\sigma_{j+1}^{\tau,\nu})\right)(t^{n+1}-,\cdot).
\end{align}
\end{itemize}
Solutions of~\eqref{eq:hom_step1} and~\eqref{eq:hom_step2} satisfy
\begin{equation}\label{eq:FL1}
    \norma{\rho_{j}^{\tau,\nu}(t^{n+1}-,\cdot) - \sigma_{j}^{\tau,\nu}(t^{n+1}-,\cdot)}_1
    \leq 
    \norma{\rho_{j}^{\tau,\nu}(t^{n}+,\cdot) - \sigma_{j}^{\tau,\nu}(t^{n}+,\cdot)}_1
    +  \mathfrak{T}\, \mathfrak{L} \,  \Delta t^\nu,
\end{equation}
see e.g.~\cite[Section 3]{Bianchini2000}. 
Indeed, by~\cite[Corollary 3.4]{HoldenRisebro2019}, we know that for all $t\geq 0$
\[
\tv(\brho^{\tau,\nu} (t,\cdot)):=\sum_{j=1}^M \tv(\rho_j^{\tau,\nu} (t,\cdot))\leq
\sum_{j=1}^M \tv(\rho_j^{0,\nu})\leq
\sum_{j=1}^M \tv(\rho_j^0) =: \tv(\brho^0).
\]
Subtracting~\eqref{eq:update2} from~\eqref{eq:update1} and setting $\rho_{j,\nu}^\tau(t^{n+1}-,\cdot):=\rho_j$
and $\sigma_{j,\nu}^\tau(t^{n+1}-,\cdot):=\sigma_j$ for $j=1,\ldots,M$, we get (using notations introduced earlier for the partial derivatives of the terms $S_j$)
\begin{align*}
    &\rho_{j}^{\tau,\nu}(t^{n+1}+,\cdot) - \sigma_{j}^{\tau,\nu}(t^{n+1}+,\cdot) \\
    = &~  \rho_{j}  - \sigma_{j}
    +  \frac{\Delta t^\nu}{\tau}
    \left( 
    S_{j-1} (\rho_{j-1},\rho_{j}) - S_{j} (\rho_{j},\rho_{j+1})
    - S_{j-1} (\sigma_{j-1},\sigma_{j}) + S_{j} (\sigma_{j},\sigma_{j+1})
    \right) \\
     = &~  \rho_{j}  - \sigma_{j}
    +  \frac{\Delta t^\nu}{\tau}
    \left( \del_1 S_{j-1} (\rho_{j-1}-\sigma_{j-1}) + \del_2 S_{j-1} (\rho_{j}-\sigma_{j})
    - \del_1 S_{j} (\rho_{j}-\sigma_{j})
    - \del_2 S_{j} (\rho_{j+1}-\sigma_{j+1})
    \right) \\
      = &~ \left( \rho_{j}  - \sigma_{j} \right)
    \left[ 1 - \frac{\Delta t^\nu}{\tau} (\del_1 S_{j} - \del_2 S_{j-1}) \right]
    +  \frac{\Delta t^\nu}{\tau}
    \left( \del_1 S_{j-1} (\rho_{j-1}-\sigma_{j-1}) 
    - \del_2 S_{j} (\rho_{j+1}-\sigma_{j+1})
    \right).
\end{align*}
Since $\del_1 S_{j-1} \geq 0$ and $\del_2 S_{j}\leq 0$ by hypothesis and $ 1 - \frac{\Delta t^\nu}{\tau} (\del_1 S_{j} - \del_2 S_{j-1})\geq 0$ for $\Delta t^\nu$ sufficiently small, we conclude that if 
$\rho^{\tau,\nu}_{j}(t^{n+1}-,\cdot)\leq \sigma^{\tau,\nu}_{j}(t^{n+1}-,\cdot)$ for all $j=1,\ldots,M$, then 
$\rho_{j}^{\tau,\nu}(t^{n+1}+,\cdot) \leq \sigma_{j}^{\tau,\nu}(t^{n+1}+,\cdot)$ for all $j=1,\ldots,M$.
Then, by the Crandall-Tartar lemma~\cite[Lemma 2.13]{HoldenRisebroBook2015}, we get
\begin{equation}\label{eq:FL2}
\sum_{j=1}^M\norma{\rho_{j}^{\tau,\nu}(t^{n+1}+,\cdot) - \sigma_{j}^{\tau,\nu}(t^{n+1}+,\cdot)}_1
    \leq 
   \sum_{j=1}^M \norma{\rho_{j}^{\tau,\nu}(t^{n+1}-,\cdot) - \sigma_{j}^{\tau,\nu}(t^{n+1}-,\cdot)}_1.
\end{equation}
Putting together~\eqref{eq:FL1} and~\eqref{eq:FL2}, we obtain the desired estimate. 
\end{proof}

\medskip

Let us now assume: 
\begin{itemize}
    \item[{\bf (V)}] $v_j=v$ (in particular, the maximal speeds become $V_j=V$) and $R_j=R/M$ the maximal density on each lane (so that $v(R/M)=0$), 
and thus $F_j(\rho)=F(\rho)=\rho v(\rho)$, for all $j=1,\ldots,M$.
\end{itemize}

\begin{lemma} \label{lem:equilibrio}
Under hypotheses {\bf (S0)}, {\bf (S1)} and {\bf (V)}, let the initial data $\rho_j^0\in \left(\L1\cap\BV\right) (\R;[0,R/M])$ for $j=1,\ldots,M$, be at equilibrium, i.e.
\begin{equation} \label{eq:ICequilibrio}
S_{j-1}\left(\rho_{j-1}^0(x),\rho_j^0(x)\right) = S_{j}\left(\rho_{j}^0(x),\rho_{j+1}^0(x)\right) = 0,
\qquad x\in\R,\quad j=1,\ldots,M.
\end{equation}
Then the solution stays at equilibrium, i.e.
\begin{equation*} 
S_{j-1}\left(\rho^\tau_{j-1}(t,x),\rho^\tau_j(t,x)\right) = S_{j}\left(\rho^\tau_{j}(t,x),\rho^\tau_{j+1}(t,x)\right) = 0,
\qquad j=1,\ldots,M,
\end{equation*}
for all $x\in\R$, $t>0$. In particular, the solution is $\L1$-Lipschitz continuous in time, uniformly in $\tau$:
\begin{equation}\label{eq:uniformLip}
    \norma{\rho^\tau_j(t_1,\cdot)-\rho^\tau_j(t_2,\cdot)}_1 \leq K \modulo{t_1-t_2}, \qquad j=1,\ldots,M,
\end{equation}
for any $t_1,t_2\in\R^+$, with $K=L\T$.

\end{lemma}

The proof is trivial observing that system~\eqref{eq:multiLWR} reduces to $M$ identical equations with the same initial datum, and therefore $\rho^\tau_1(t,\cdot)=\ldots=\rho^\tau_M(t,\cdot):=\bar\rho^\tau$ for all $t>0$.

\medskip


The uniform Lipschitz estimate~\eqref{eq:uniformLip} and the uniform $\tv$ bound~\eqref{eq:TVbound} allow to apply Helly's Theorem to the sequence $\left\{\brho^\tau \right\}_{\tau>0}$, proving the existence of a subsequence converging in $\Lloc1$ (indeed, the sequence is constant for all $\tau>0$, so there is nothing to prove). In particular, the limit $\bar\brho= \left(\bar\rho_1,\ldots,\bar\rho_M\right)$ satisfies $\bar\rho_j=\bar\rho^\tau$ for any $\tau>0$ and it holds
\[
\del_t\bar\rho_j +\del_x F(\bar\rho_j) = 0, \qquad 
j=1,\ldots,M,
\]
so that, setting $r:=\sum_{j=1}^M\bar\rho_j=M\bar\rho^\tau$, $r$ is a entropy weak solution to the scalar Cauchy problem
\begin{equation}\label{eq:limitLWR}
\begin{cases}
\del_t r +\del_x f(r) =0, & t>0, ~x\in\R, \\
r(0,x) = \sum_{j=1}^M\rho_j^0 (x), & x\in\R,
\end{cases}   
\end{equation}
where we have set $f(r):= M F(r/M)$.

\medskip
We can now give the following result about the convergence of~\eqref{eq:multiLWR} to~\eqref{eq:limitLWR} as $\tau\searrow 0$.

\begin{theorem} \label{thm:relax}
 Under hypotheses {\bf (S0)}, {\bf (S1)} and {\bf (V)}, let $\rho_j^0\in \left(\L1\cap\BV\right) (\R;[0,R/M])$ and $\brho^\tau=\left(\rho^\tau_1,\ldots,\rho^\tau_M\right)$ be the corresponding solution of~\eqref{eq:multiLWR}.
 Then, for each $t>0$,
 \[
 \rho^\tau_j (t,\cdot) \longrightarrow \frac{r(t,\cdot)}{M} \qquad\hbox{in } \Lloc1(\R;\R)
 \]
 as $\tau\searrow 0$, where $r \in \C{0} \left(\R^+; \L1 \left(\R; [0,R]\right)\right)$ is the entropy weak solution of~\eqref{eq:limitLWR}.
\end{theorem}

\begin{proof}
We adapt an argument by Bressan and Shen~\cite[Section 5]{BressanShen2000}.
Denote by $\tilde\brho^\tau=\left(\tilde\rho^\tau_1,\ldots,\tilde\rho^\tau_M\right)$ the solution of 
\begin{equation} \label{eq:ode}
\del_t \rho_j = \dfrac{1}{\tau}\left(S_{j-1} (\rho_{j-1},\rho_j) - S_{j} (\rho_{j},\rho_{j+1})\right), \qquad j=1,\ldots,M,
\end{equation}
with the same initial data $\rho_j^0$. Linearizing~\eqref{eq:ode} around the equilibrium initial condition $\brho(0,x)=\left(\rho_1(0,x),\ldots,\rho_M(0,x)\right)$ with $\rho_j(0,x):= r(0,x)/M$ for $j=1,\ldots,M$ we obtain
\begin{equation*} 
\del_t \brho = \dfrac{1}{\tau}\left(\mathbb{S}(\brho(0,\cdot)) + D\mathbb{S}(\brho(0,\cdot)) (\brho -\brho(0,\cdot)\right),
\end{equation*}
where we have set $\mathbb{S}_j (\brho(0,\cdot)) := S_{j-1} (r(0,\cdot)/M,r(0,\cdot)/M) - S_{j} (r(0,\cdot)/M,r(0,\cdot)/M)=0$ and 
\[
D\mathbb{S}(\brho) =
\left[
\begin{array}{cccccc}
  -\del_1 S_1   & -\del_2 S_1 & 0 & \cdots   & \cdots   & 0\\
  \del_1 S_1   & \del_2 S_1 - \del_1 S_2 & -\del_2 S_2 & 0 & \cdots     & \vdots\\
  0 & \del_1 S_2 & \del_2 S_2 - \del_1 S_3 & -\del_2 S_3 & 0 & \vdots \\
  \vdots & 0 & \del_1 S_3 & \ddots & &  \\
  0 & \cdots  & \cdots & 0 & \del_1 S_{M-1} & \del_2 S_{M-1}
\end{array}
\right]
\]
is a tridiagonal matrix such that
\[
D\mathbb{S}(\brho)_{i,j} = 
\begin{cases}
\del_1 S_{j} & i= j+1, \\
\del_2  S_{j-1} - \del_1  S_{j} & i=j,\\
 - \del_2  S_{j-1} & i=j-1.
\end{cases}
\]
The similarity transformation to a symmetric matrix $\mathbb{A}(\brho)$
gives
\[
\mathbb{A}(\brho)_{i,j} = 
\begin{cases}
 \sqrt{-\del_1 S_{j} \del_2  S_{j}} & i=j+1, \\
\del_2  S_{j-1} - \del_1  S_{j} & i=j,\\
 \sqrt{-\del_1 S_{j-1} \del_2  S_{j-1}} & i=j-1,
\end{cases}
\]
which, by the monotonicity hypotheses in ${\bf (S)}$, turns out to be negative definite (this can be easily seen applying Sylvester's criterion).
Therefore,
\begin{equation} \label{eq:linear_estimate}
    \sum_{j=1}^M \modulo{\tilde\rho^\tau_j(t,x)-\frac{r(0,x)}{M}} = \mathcal{O}(1) \cdot e^{-t/\tau}, \qquad x\in\R\,.
\end{equation}
Moreover, for any interval $[a,b]\subset \R$, comparison between~\eqref{eq:multiLWR} and~\eqref{eq:ode} gives, by Lemma~\ref{lem:multiLWRstability},
\begin{equation} \label{eq:flux_stability}
    \sum_{j=1}^M \norma{\tilde\rho^\tau_j(t,\cdot)-\rho^\tau_j(t,\cdot)}_{\L1([a,b])} = \mathcal{O}(1) \cdot (b-a)\, t\,.
\end{equation}
Taking $t=\sqrt{\tau}$ in~\eqref{eq:linear_estimate} and ~\eqref{eq:flux_stability} we get
\begin{equation*} 
    \sum_{j=1}^M \norma{\rho^\tau_j(\sqrt{\tau},\cdot)-\frac{r(0,\cdot)}{M}}_{\L1([a,b])} = \mathcal{O}(1) \cdot (b-a) \left(\sqrt{\tau} + e^{-1/\sqrt{\tau}}\right).
\end{equation*}
Taking $\brho^\tau(\sqrt{\tau},\cdot)$ as initial datum in~\eqref{eq:multiLWR}, the stability estimate~\eqref{eq:L1contraction} gives
\begin{equation*} 
    \sum_{j=1}^M \norma{\rho_j^\tau(t,\cdot) - \frac{r(t-\sqrt{\tau},\cdot)}{M}}_{\L1([a,b])}
    = \mathcal{O}(1) \cdot (b-a) \left(\sqrt{\tau} + e^{-1/\sqrt{\tau}}\right)
\end{equation*}
for all $t\geq \sqrt{\tau}$. By the $\L1$-Lipschitz continuity in time of $r$ guaranteed by Lemma~\ref{lem:equilibrio}, we finally get
\begin{equation*} 
    \sum_{j=1}^M \norma{\rho_j^\tau(t,\cdot) - \frac{r(t,\cdot)}{M}}_{\L1([a,b])}
    = \mathcal{O}(1) \cdot (b-a) \left(\sqrt{\tau} + e^{-1/\sqrt{\tau}}\right)
\end{equation*}
for all $t>0$.
\end{proof}

\section{Asymptotic behaviour for $M=2$ with AV}
\label{sec:relaxAV}

We now focus on the case of a two-lane road with one AV, say on lane $i=1$ (without loss of generality by symmetry), which reads:
\begin{equation} \label{eq:2LWR+AV}
\begin{cases}
\del_t \rho_1 + \del_x F_1(\rho_1) = - \dfrac{1}{\tau} S (\rho_{1},\rho_{2}),  \\[5pt]
\del_t \rho_2 + \del_x F_2(\rho_2) =  \dfrac{1}{\tau} S (\rho_{1},\rho_{2}),  \\[5pt]
\rho_j(0,x) = \rho^0_j(x),\qquad j=1,2,\\[5pt]
\dot y (t) = \min\left\{u (t), v_1(\rho_1(t,y(t)+))\right\}, \\
y (0) = y_0,   \\
\rho_1 (t,y(t)) \left(v_1(\rho_i  (t,y (t))) - \dot y (t) \right) \leq  0,
\end{cases}
\qquad 
x\in\R, ~t\geq 0,
\end{equation}
where we have set $S(\rho_1,\rho_2):=S_1(\rho_1,\rho_2)$. 

In this case, Lemma~\ref{lem:timeLip} becomes:

\begin{lemma} \label{lem:timeLipM2}
Let {\bf (S0)} and {\bf (S2)} hold and $\Delta t^\nu \leq \tau/2\Sim$, then  for any $t_1 \leq t_2\in [0,T]$, $T>0$, the approximate solutions $\brho^{\tau,\nu}$ satisfy 
\begin{equation} \label{eq:Lip_est}
\sum_{j=1}^2\norma{\rho_{j}^{\tau,\nu}(t_1,\cdot) - \rho_{j}^{\tau,\nu}(t_2,\cdot)}_1 \leq \left(C+C_\tau(t_1)\right)\modulo{t_1-t_2} ,   \qquad\hbox{for all }\nu\in\N,
\end{equation}
where $C=2L\T(1+\Sim/c)$ is independent of $\tau$  and  
\[
C_\tau(t_1)=\frac{4\Sim}{\tau} \exp
     \left(-\frac{2c}{\tau}t_1 \right) \sum_{j=1}^2 \norma{\rho_j^0}_1 .
\]
\end{lemma}

\begin{proof}
We follow the same procedure and notations as in the proof of Lemma~\ref{lem:timeLip}.
Setting
\[
g_{1,i}^\pm = g_{2,i}^\pm := \norma{S (\rho_{1,\nu}^\tau,\rho_{2,\nu}^\tau)(t^{i}\pm,\cdot)}_1 ,
\]
we compute
\begin{align*}
  g_{1,i}^+ + g_{2,i}^+&=  2 \norma{S(\rho_1^+,\rho_2^+)}_1 \\
  &=  2 \norma{S(\rho_1^-,\rho_2^-) +\del_1S \cdot (\rho_1^+ - \rho_1^-) + \del_2S \cdot (\rho_{2}^+ - \rho_{2}^-) }_1 \\
  &=  2 \norma{S(\rho_1^-,\rho_2^-) - \frac{\Delta t^\nu}{\tau} \del_1S\cdot S(\rho_1^-,\rho_2^-) + \frac{\Delta t^\nu}{\tau} \del_2S\cdot S(\rho_1^-,\rho_2^-) }_1 \\
  &= 2 \left(1 -\frac{\Delta t^\nu}{\tau} \left(\del_1 S -\del_2S\right)\right)\norma{S(\rho_1^-,\rho_2^-)}_1 \\
  &= \left(1 -\frac{\Delta t^\nu}{\tau} \left(\del_1 S -\del_2S\right)\right)
  \left( g_{1,i}^- + g_{2,i}^- \right) \\
  &
  \leq \left(1 -\frac{2c}{\tau} \Delta t^\nu \right)
  \left( g_{1,i}^- + g_{2,i}^- \right).
\end{align*}
Moreover, for $j=1,2$, it holds
\begin{align*}
  g_{j,i}^- - g_{j,i-1}^+
  &= \norma{S(\rho_{1,i}^-,\rho_{2,i}^-)}_1 - \norma{S(\rho_{1,i-1}^+,\rho_{2,i-1}^+)}_1 \\
  &\leq \norma{S(\rho_{1,i}^-,\rho_{2,i}^-) - S(\rho_{1,i-1}^+,\rho_{2,i-1}^+)}_1 \\
  &= \left\|\del_1 S^i \cdot ( \rho_{1,i}^- -\rho_{1,i-1}^+) + 
  \del_2 S^i \cdot ( \rho_{2,i}^- -\rho_{2,i-1}^+)\right\|_1 \\
  &\leq 2\Sim L\T \Delta t^\nu,
\end{align*}
where we have set $\del_l S^i:=\del_l S(s_1^i,s_2^i)$, $l=1,2$, with
$s_1^i= \rho_{1,i-1}^+ + \xi (\rho_{1,i}^- - \rho_{1,i-1}^+)$ and $s_2^i= \rho_{2,i-1}^+ + \xi (\rho_{2,i}^- - \rho_{2,i-1}^+)$ for some $\xi\in[0,1]$.\\
Therefore, we get
\begin{align*}
    \sum_{j=1}^M&\norma{\rho_{j}^{\tau,\nu}(t_2,\cdot) - \rho_{j}^{\tau,\nu}(t_1,\cdot)}_1
    \leq 2 L \T \left( t_2 - t_1\right)
    +  \frac{\Delta t^\nu}{\tau}
    \sum_{i=k}^{k+N} \sum_{j=1}^2 g_{j,i}^- \\
    &\leq 2 L \T \left( t_2 - t_1\right)
    +  \frac{\Delta t^\nu}{\tau}
    \sum_{i=k}^{k+N} \left[\left(1-\frac{2c}{\tau}\Delta t^\nu \right)^{i-1}\sum_{j=1}^2 g_{j,0}^+ + 4\Sim L \T \Delta t^\nu \sum_{\ell=0}^{i-1}\left(1-\frac{2c}{\tau}\Delta t^\nu \right)^{\ell}\right]\\
    &\leq 2 L \T \left( t_2 - t_1\right)
    +  \frac{\Delta t^\nu}{\tau}
    \sum_{i=k}^{k+N} \left[\left(1-\frac{2c}{\tau}\Delta t^\nu \right)^{i-1}\sum_{j=1}^2 g_{j,0}^+ + 2\Sim L \T  \frac{\tau}{c} \right]\\
    &\leq 2 L \T \left( t_2 - t_1\right)
    +  2\Sim L \T   \frac{t_2 - t_1+\Delta t^\nu}{c} + \frac{\Delta t^\nu}{\tau}
    \sum_{i=k}^{k+N} \left[\left(1-\frac{2c}{\tau}\Delta t^\nu \right)^{i-1}\sum_{j=1}^2 g_{j,0}^+  \right]\\
    &\leq 2 L \T \left( t_2 - t_1+\Delta t^\nu\right)\left( 1+\frac{\Sim}{c} \right)
     + \frac{1}{2c }
     \left(1-\frac{2c}{\tau}\Delta t^\nu \right)^{k-1} \left[1-\left(1-\frac{2c}{\tau}\Delta t^\nu \right)^{N+1}\right] \sum_{j=1}^2 g_{j,0}^+ \\
     &\leq 2 L \T \left( t_2 - t_1+\Delta t^\nu\right)\left( 1+\frac{\Sim}{c} \right)
     + \frac{1}{c }
     \left(1-\frac{2c}{\tau}\Delta t^\nu \right)^\frac{k\Delta t^\nu}{\Delta t^\nu} (N+1)\frac{2c}{\tau}\Delta t^\nu  \sum_{j=1}^2 g_{j,0}^+ \\
     &\leq  \left( t_2 - t_1+\Delta t^\nu\right)\left[2 L \T\left( 1+\frac{\Sim}{c} \right)
     +  \frac{4\Sim}{\tau} \exp
     \left(-\frac{2c}{\tau}t_1 \right) \sum_{j=1}^2 \norma{\rho_j^0}_1 \right],
\end{align*}
where we have used the inequality 
$\left(1-\frac{2c}{\tau}\Delta t^\nu \right)\geq\frac{1}{2}$, which holds provided
that $\Delta t^\nu\leq \tau/4c$.
\end{proof}

\begin{remark} \label{rem:uniformLip}
By Lemma~\ref{lem:timeLipM2}, the Lipschitz constant $C_\tau(t_1)$ is uniformly bounded as $\tau\searrow 0$ for any strictly positive time. Indeed, taking $\frac{1}{n}\leq t_1\leq t_2$ with $n\in\N$ in~\eqref{eq:Lip_est}, we have 
\begin{equation} \label{eq:bound}
C_\tau(t_1)\leq\frac{4\Sim}{\tau} \exp
     \left(-\frac{2c}{n \tau}\right) \sum_{j=1}^2 \norma{\rho_j^0}_1 ,
\end{equation}
which goes to zero with $\tau$.

Note also that, under the hypotheses of Lemma~\ref{lem:timeLipM2}, if the initial datum is at equilibrium, i.e. $S(\rho_1^0(x),\rho_2^0(x))=0$ for a.e. $x\in\R$, it holds $C_\tau (t) = 0$ for all $t\geq 0$. Therefore, the approximate solutions $\brho^{\tau,\nu}$ are uniformly Lipschitz continuous for all $t\geq 0$:
\begin{equation} \label{eq:Lip_est_uniform}
\sum_{j=1}^2\norma{\rho_{j}^{\tau,\nu}(t_1,\cdot) - \rho_{j}^{\tau,\nu}(t_2,\cdot)}_1 \leq C\modulo{t_1-t_2} ,   \qquad\hbox{for all }\nu\in\N, \tau>0.
\end{equation}

\end{remark}

We can now prove the following convergence result.

\begin{theorem} \label{thm:relaxM2}
 Under hypotheses {\bf (S0)}, {\bf (S1)} and {\bf (S2)}, let $v_1=v_2=v$  
 (thus $F_1(\rho)=F_2(\rho)=F(\rho)=\rho v(\rho)$)
 and 
 $\rho_1^0,~\rho_2^0\in \left(\L1\cap\BV\right) (\R;[0,R/2])$ and $\brho^\tau=\left(\rho^\tau_1,\rho^\tau_2\right)$ and $y^\tau$ be the corresponding entropy weak solution of~\eqref{eq:2LWR+AV}.
 Then 
 \begin{align*}
 &\rho^\tau_j \longrightarrow \frac{r}{2} \qquad\hbox{in } \Lloc1(]0,+\infty[\,\times\R;\R), \quad i=1,2, \\
& y^\tau \longrightarrow y \qquad\hbox{in }
\Lloc{\infty}(\R^+;\R),
  \end{align*}
 as $\tau\searrow 0$, where $y\in\C0 (\R^+;\R)$ and $r \in \C{0} \left(\R^+; \L1 \left(\R; [0,R]\right)\right)$
is a weak solution of~\eqref{eq:limitLWR} with $M=2$. 
If moreover $\dot y^\tau \longrightarrow \dot y$ in 
$\Lloc{1}(\R^+;\R)$, then $(r,y)$
satisfies
 \begin{align}
    \int_{\R^+}\int_\R &\left(
    |r-k| \del_t \varphi +\sgn (r-k)(f(r)-f(k)) \del_x\varphi
    \right) dx\, dt  \nonumber\\
    & {+} \, {\int_{\R^+} \left(
    f(k)  - \dot y (t) k \right)^+ \varphi(t,y(t)) \, dt} \geq 0\,, \label{eq:non-classicalterm}
\end{align}
with $f(r):= 2 F(r/2)$. 
\end{theorem}

\begin{remark}
We notice that, as pointed out in the proof of Theorem~\ref{thm:multiLWR+AV}, we need to assume a-priori the $\W{1,1}$ convergence of the AV trajectories $y^\nu$ to $y$. \\
    Theorem~\ref{thm:relaxM2} shows in particular that $(r,y)$ cannot be an entropy weak solution of the moving bottleneck model~\cite{DMG2014,GGLP2019} with $\alpha=1/2$, which writes
 \begin{equation}  \label{eq:MB2}
 \begin{cases}
 \del_t r  + \del_x f(r) = 0, \\
 r(0, x) = \rho_1^0(x)+\rho_2^0(x), \\
{\dot y (t) = \min\left\{u (t), \tilde v(r(t,y(t)+))\right\},} \\
 y (0) = y_0,   \\
 {f(r  (t,y (t))) - \dot y (t)\, r(t,y(t))  \leq  \mathcal{F}_{1/2}\left(\dot y(t)\right)
 :=\max_{r\in[0,R]} \left( \frac{1}{2} f(2r)-r\,\dot y (t)\right),}
 \end{cases}
 x\in\R, ~t\geq 0,
 \end{equation}
 where $\tilde v (r) := f(r)/r=v(r/2)$, whose entropy weak solution satisfies
\begin{align}
    \int_{\R^+}\int_\R &\left(
    |r-k| \del_t \varphi +\sgn (r-k)(f(r)-f(k)) \del_x\varphi
    \right) dx\, dt  \nonumber\\
    & {+} \, { 2 \int_{\R^+} \left(
    f(k)  - k\,\dot y (t)  - \mathcal{F}_{1/2}\left(\dot y(t)\right)\right)^+ \varphi(t,y(t)) \, dt} \geq 0\,. \label{eq:non-classical}
\end{align}
Observe that we can estimate
\begin{align*}
   2& \left(
    f(k)  - k\, \dot y (t)   - \mathcal{F}_{1/2}\left(\dot y(t)\right)\right)^+\\
    &= 2\left(
    f(k)  - k\, \dot y (t)   - \max_{r\in[0,R]} \left( \frac{1}{2} f(2r)-r\,\dot y (t)\right)\right)^+\\
    &= \left(
    2 f(k)  - 2k\, \dot y (t)  - \max_{r\in[0,R]} \left( f(2r)-2r\,\dot y (t)\right)\right)^+\\
    &= \left(
    2 f(k)  - 2k\, \dot y (t)  + \min_{r\in[0,R]} \left( -f(2r)+2r\,\dot y (t)\right)\right)^+\\
    &\leq \left(
    2 f(k)  - 2k\, \dot y (t)  + \left( -f(k)+k\,\dot y (t)\right)\right)^+ \\
    &= \left(
    f(k)  - k\, \dot y (t)  \right)^+
\end{align*}
taking $r=k/2$. This shows that~\eqref{eq:non-classical} implies~\eqref{eq:non-classicalterm}, but the converse is not true. See also the numerical examples in Section~\ref{sec:relaxAVnum}.

\end{remark}

\begin{proofof}{Theorem~\ref{thm:relaxM2}}
 By Helly's Theorem, for any $n\in\N$ there exists a subsequence $\{\tau_n\}\subseteq\{\tau_{n-1}\}$ such that $\brho^{\tau_n}$ converges to some function $\bar\brho$ in $\Lloc1([1/n,+\infty[\,\times\R)$.
We can then construct by diagonal process a sequence, still labeled $\{\brho^{\tau_n}\}$, converging to $\bar\brho$ in $\Lloc1(]0,+\infty[\,\times\R)$. From~\eqref{eq:Lip_est} and~\eqref{eq:bound}, the limit satisfies
\begin{equation} \label{eq:uniformLip2}
    \sum_{j=1}^2\norma{\bar\rho_{j}(t_1,\cdot) - \bar\rho_{j}(t_2,\cdot)}_1 \leq C\modulo{t_1-t_2} ,   \qquad\hbox{for any }t_1,t_2>0.
\end{equation}
From~\eqref{eq:multiEC} we deduce that
\[
S(\bar\rho_1,\bar\rho_2)=0\qquad\hbox{a.e. in }]0,+\infty[\,\times\R.
\]
By {\bf (S1)}, this implies $\bar\rho_1=\bar\rho_2= \bar\rho$.
Moreover, since $\{\brho^{\tau_n}\}$ is equi-bounded, it actually converges in $\Lloc1([0,+\infty[\,\times\R)$.
By~\eqref{eq:uniformLip2} we can then pass to the limit in $\Lloc1$ concluding that
\[
\lim_{t\to 0+} S(\bar\rho_1(t,\cdot),\bar\rho_2(t,\cdot)) =0\,.
\]
Summing~\eqref{eq:multiEC} for $j=1,2$ and $\tau=\tau_n$ with $\varphi\in\Cc1 (]0,+\infty[\,\times\R;\R^+)$,
we obtain
\begin{align}
    \int_{\R^+}&\int_\R \left(
    |\rho^\tau_1-\kappa|+
     |\rho^\tau_2-\kappa|
      \right) \del_t\varphi \, dx\, dt \nonumber \\
    & +  \int_{\R^+}\int_\R
    \left(
    \sgn (\rho^\tau_1-\kappa)\left(F(\rho^\tau_1)-F(\kappa)\right) 
    + \sgn (\rho^\tau_2-\kappa)\left(F(\rho^\tau_2)-F(\kappa)\right) 
    \right) \del_x\varphi\, dx\, dt \nonumber\\ 
    & + \dfrac{1}{\tau} \int_{\R^+}\int_\R \left(\sgn (\rho^\tau_2-\kappa)- \sgn(\rho^\tau_1-\kappa) \right) S_{j} (\rho^\tau_{1},\rho^\tau_{2}) \, \varphi \, dx\, dt \nonumber \\
   & {+} \, 2  \int_{\R^+} \left(
    F(\kappa)  - \dot y^\tau (t) \kappa \right)^+ \varphi(t,y^\tau(t)) \, dt \geq 0\, \label{eq:ECM=2}
\end{align}
for all $\kappa\in\R$.
We note that
\begin{align*}
\left(\sgn (\rho^\tau_2-\kappa)- \sgn(\rho^\tau_1-\kappa) \right) S_{j} (\rho^\tau_{1},\rho^\tau_{2})
\leq 0\qquad \hbox{for every}~\kappa\in\R. 
\end{align*}
Indeed, this is obviously true if $\sgn (\rho^\tau_2-\kappa) = \sgn(\rho^\tau_1-\kappa)$. If
$\rho^\tau_1 \leq \kappa \leq \rho^\tau_2$, we get
\[
\left(\sgn (\rho^\tau_2-\kappa)- \sgn(\rho^\tau_1-\kappa) \right) S_{j} (\rho^\tau_{1},\rho^\tau_{2})
=
2 S_{j} (\rho^\tau_{1},\rho^\tau_{2}) \leq S_j(\kappa,\kappa) =0\,,
\]
due to the monotonicity in ${\bf (S0)}$ and ${\bf (S1)}$. The same holds for $\rho^\tau_2 \leq \kappa \leq \rho^\tau_1$. \\
Therefore,~\eqref{eq:ECM=2} implies
\begin{align}
    \int_{\R^+}&\int_\R \left(
    |\rho^\tau_1-\kappa|+
     |\rho^\tau_2-\kappa|
      \right) \del_t\varphi \, dx\, dt \nonumber \\
    & +  \int_{\R^+}\int_\R
    \left(
    \sgn (\rho^\tau_1-\kappa)\left(F(\rho^\tau_1)-F(\kappa)\right) 
    + \sgn (\rho^\tau_2-\kappa)\left(F(\rho^\tau_2)-F(\kappa)\right) 
    \right) \del_x\varphi\, dx\, dt \nonumber\\ 
    & {+} \, 2  \int_{\R^+} \left(
    F(\kappa)  - \dot y^\tau (t) \kappa \right)^+ \varphi(t,y^\tau(t)) \, dt \geq 0\,. \label{eq:ECM=2bis}
\end{align}
Passing to the limit as $\tau_n\searrow 0$ in~\eqref{eq:ECM=2bis}, thanks to the hypothesis of $\L1$ convergence of $\dot y^\nu$ to $\dot y$, we get that
$\bar\rho$ satisfies
\begin{align*}
    \int_{\R^+}\int_\R &\left(
    |2\bar\rho-2\kappa| \del_t \varphi +\sgn (\bar\rho-\kappa)(2F(2\bar\rho/2)-2F(2\kappa/2)) \del_x\varphi
    \right) dx\, dt  \\
    & {+} \,  \int_{\R^+} \left(
    2F(2\kappa/2)  - \dot y (t) 2\kappa \right)^+ \varphi(t,y(t)) \, dt \geq 0\,,
\end{align*}
which, setting $r=2\bar\rho$, $k=2\kappa$ and $f(r)= 2 F(r/2)$, is equivalent to~\eqref{eq:non-classicalterm}.

Finally, taking $\kappa\not\in [0,R/2]$ in~\eqref{eq:multiEC} we get for $j=1,2$ and $\tau=\tau_n$,
\begin{align*}
    \int_{\R^+}\int_\R &\left(
    \rho^\tau_j \del_t \varphi +F(\rho^\tau_j) \del_x\varphi
    \right) dx\, dt 
    +\int_\R \rho_j^0 (x)\varphi(0,x) \, dx  \\
    & + \dfrac{1}{\tau} \int_{\R^+}\int_\R  \left(  S_{j-1} (\rho^\tau_{j-1},\rho^\tau_j) - S_{j} (\rho^\tau_{j},\rho^\tau_{j+1})   \right) \varphi \, dx\, dt  = 0\,
\end{align*}
for all $\varphi\in\Cc1 (\R^2;\R)$.
Passing to the limit as $\tau_n\searrow 0$ and summing over $j=1,2$, we get
\begin{align*}
    &\int_{\R^+}\int_\R \left(
    2\bar\rho\, \del_t \varphi +2F(2\bar\rho/2) \del_x\varphi
    \right) dx\, dt 
    +\int_\R \left(\rho_1^0 (x)+\rho_2^0 (x)\right)\varphi(0,x) \, dx  \\
    &= \int_{\R^+}\int_\R \left(
    r\, \del_t \varphi + f(r) \del_x\varphi
    \right) dx\, dt 
    +\int_\R \left(\rho_1^0 (x)+\rho_2^0 (x)\right)\varphi(0,x) \, dx = 0\,,
\end{align*}
concluding the proof. 
\end{proofof}

\section{Numerical simulations} \label{sec:num}

To illustrate the behaviour of solutions of model~\eqref{eq:multiLWR+manyAV}, we compute approximate solutions obtained via a finite volume scheme with time splitting developed merging the numerical schemes proposed in~\cite{GoatinRossi2020} for multi-lane models and in~\cite{CDMG2017} for correctly capturing the dynamics around moving bottlenecks. See also~\cite{GDDMF2023, GoatinRossiNHM2020, VGC2017} for further extensions.

We consider a uniform space mesh of width $\Delta x$ and a time step $\Delta t$, subject to the stability condition 
\begin{equation*}
    \Delta t \leq \min\left\{
    \frac{\Delta x}{\|F_j'\|_\infty},
    \frac{\tau}{2\Sim}
    \right\},
\end{equation*}
see Lemma~\ref{lem:Linfty}.
We set
$
  x_k =  \left(k+1/2\right) \Delta x$ and
    $x_{k-1/2} = k \, \Delta x$
for $k\in \Z$, 
where $x_k$ is the center of cell $C_k=[x_{k-1/2},x_{k+1/2}[$  and $x_{k\pm 1/2}$ are the
interfaces. Set $\lambda= \Delta t/\Delta x$.  \\
We approximate the initial data as  
\begin{equation*}
  \rho_{j,k}^0 = \frac{1}{\Delta x} \int_{x_{k-1/2}}^{x_{k+1/2}} \rho_{j}^0 (x)\,  dx,
  \qquad \hbox{for}~j=1, \ldots, M.
\end{equation*}
The approximate solutions to~\eqref{eq:multi_model} are obtained through a Godunov type scheme with fractional step accounting for the contribution of the source term: for $j=1,\ldots,M$, and $n\in\N$, we set
\begin{align}
  \rho_{j,k}^{n+1/2} = \
  & \rho_{j,k}^n - \lambda\left[
    \Flux_j (\rho_{j,k}^n,\rho_{j,k+1}^n) - \Flux_j (\rho_{j,k-1}^n,\rho_{j,k}^n)
    \right], \label{eq:Godunov}
  \\[5pt]
  \rho_{j,k}^{n+1} = \
  & \rho_{j,k}^{n+1/2} + \frac{\Delta t}{\tau} \left[ S_{j-1} (\rho_{j-1,k}^{n+1/2}, \rho_{j,k}^{n+1/2})
    -  S_{j} (\rho_{j,k}^{n+1/2}, \rho_{j+1,k}^{n+1/2})\right],\nonumber
\end{align}
where
\begin{equation*}
  \Flux_j (u,w) =
    \min\left\{D_j (u), S_j (w)\right\}
\end{equation*}
with
\begin{align*}
  D_j (u) = F_j(\min\{u,\theta_j\}), \qquad
  S_j (u) = F_j(\max\{u,\theta_j\}),
\end{align*}
$\theta_j\in [0,R_j]$ being the point of maximum of $F_j$.

To capture the presence of an AV, let's say at position $y_i^n\in C_m$ in lane $i\in\{1,\ldots,M\}$, we modify the Godunov fluxes at the interfaces $x_{m\pm 1/2}$ as follows.
We denote by $\Riem_i$ the standard Riemann solver for 
$\del_t\rho + \del_x F_i(\rho) =0$,
i.e. $\Riem_i(\rho_L,\rho_R)(\xi)$ gives the value of the self-similar entropy weak solution corresponding to the Riemann initial data $(\rho_L,\rho_R)$ at $x=\xi t$.
We also set
\begin{equation*}
  u_i^n = \frac{1}{\Delta t} \int_{t^n}^{t^{n+1}} u_i (t)\,  dt,
  \qquad \hbox{for}~n\in\N,
\end{equation*}
the discretization of the AV's desired speed.
If
\begin{equation} \label{eq:dicrete_constraint}
F_i(\Riem(\rho_{i,m-1}^n,\rho_{i,m+1}^n)(u_i^n))
>
u_i^n \, \Riem(\rho_{i,m-1}^n,\rho_{i,m+1}^n)(u_i^n),
\end{equation}
we expect the flux constraint to be active at 
$\bar x_m=x_{m-1/2} + d_m^n \Delta x$ with
$d_m^n=\rho_{i,m}^n/ \hat\rho_{u_i^n}$
to ensure mass conservation. If $d_m^n\in [0,1]$, then $\bar x_m\in C_m$ and, setting $\Delta t_m^n:=\Delta x (1-d_m^n)/u_i^n$, we replace the numerical fluxes in~\eqref{eq:Godunov} by
\begin{align*}
    \tilde\Flux_i(\rho_{i,m-1}^n,\rho_{i,m}^n) &:=
    \Flux_i(\rho_{i,m-1}^n,\hat\rho_{u_i^n}), \\
    \tilde\Flux_i(\rho_{i,m}^n,\rho_{i,m+1}^n) &:=
    \max\left\{1-\Delta t_m^n/\Delta t,0\right\} F_i(\hat\rho_{u_i^n}).
\end{align*}
Besides, to track the AV's trajectory, at each time step we implement an explicit Euler scheme:
\[
y_i^{n+1}=y_i^n+ 
\begin{cases}
    u_i^n\, \Delta t, & \hbox{if~\eqref{eq:dicrete_constraint} holds},\\
    v_i(\rho_m^n) \, \Delta t, & \hbox{otherwise}.
\end{cases}
\]
In the following numerical tests we consider~\eqref{eq:multiLWR+manyAV} with
\begin{subequations} \label{eq:function_choice}
\begin{align}
v_j(\rho_j)&=V_j (1-\rho_j/R_j),\\
    F_j(\rho_j)&=\rho_j v_j(\rho_j)\\
    S_{j} (U_{j},U_{j+1})&= \left( v_{j+1}(\rho_{j+1}) - v_{j}(\rho_{j})\right)^+\rho_j - \left( v_{j+1}(\rho_{j+1}) - v_{j}(\rho_{j})\right)^-\rho_{j+1} .    
\end{align}
\end{subequations}

\subsection{A 3-lane road with heterogeneous speed limits and AVs}
\label{sec:3lanes}

As an illustration of the behaviour of model~\eqref{eq:multiLWR+manyAV}, we consider $M=3$ and $N_j=1$ for $j=1,2,3$.
We set $R_j=1$ for $j=1,2,3$ and $V_1=50$~km/h, $V_2=80$~km/h, $V_3=100$~km/h, thus considering different speed limits on each lane. AVs' desired speeds are set constantly to $u_1^1=u_2^1=u_3^1=30$~km/h and the relaxation parameter $\tau=0.05~\hbox{km}^{-1}$. \\
We consider a road stretch of length $L=10$~km with $\Delta x=0.02$ and initial conditions
\begin{align}
    \rho_1^0(x)&= 0.5 + 0.5 \sin{(0.5\pi x)}, \nonumber\\
    \rho_2^0(x)&= 0.5 + 0.5 \cos{(0.5\pi x)}, \label{eq:oscillatingIC}\\
    \rho_1^0(x)&= 0.5 + 0.5 \sin{(\pi x)}, \nonumber
\end{align}
for the (normalized) traffic densities on each lane, and
\[
y_{0,1}^1=1,\quad y_{0,2}^1=2,\quad y_{0,3}^1=3,
\]
for the AVs' positions.
\begin{figure}[ht]
\begin{subfigure}{0.5\textwidth}
  \centering
  \includegraphics[width=\linewidth]{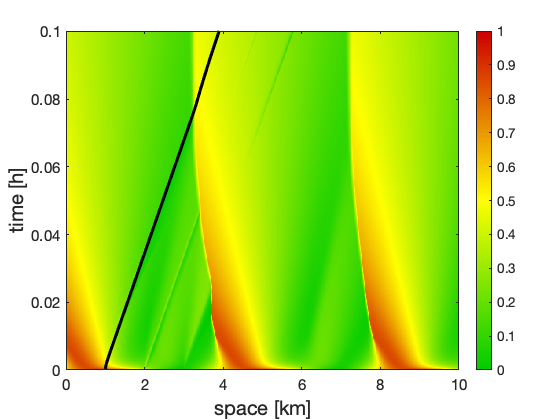} 
  \caption{Density $\rho_1$ and trajectory $y_1^1$ on lane 1}
  \label{fig:rho1}
\end{subfigure} 
\begin{subfigure}{0.5\textwidth}
  \centering
  \includegraphics[width=\linewidth]{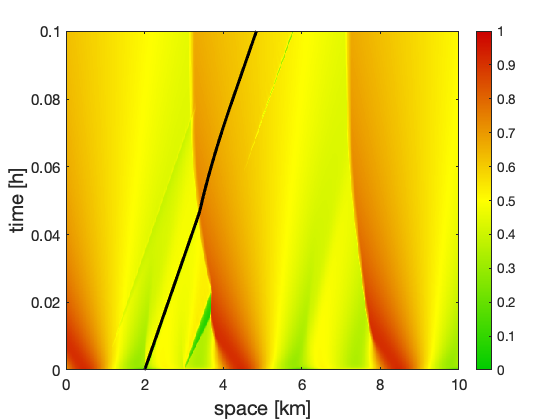} 
 \caption{Density $\rho_2$ and trajectory $y_2^1$ on lane 2}
  \label{fig:rho2} 
\end{subfigure} \\
\begin{subfigure}{0.5\textwidth}
  \centering
  \includegraphics[width=\linewidth]{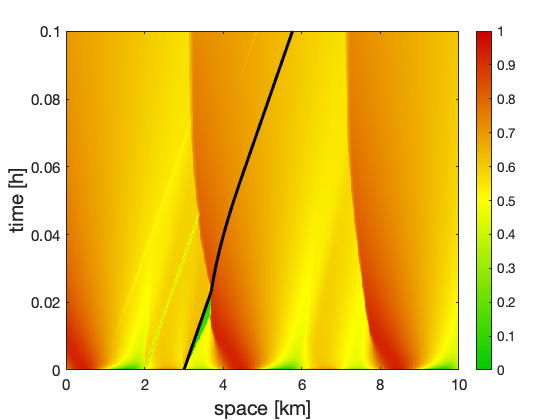} 
  \caption{Density $\rho_3$ and trajectory $y_3^1$ on lane 3}
  \label{fig:rho3}
\end{subfigure} 
\begin{subfigure}{0.5\textwidth}
  \centering
  \includegraphics[width=\linewidth]{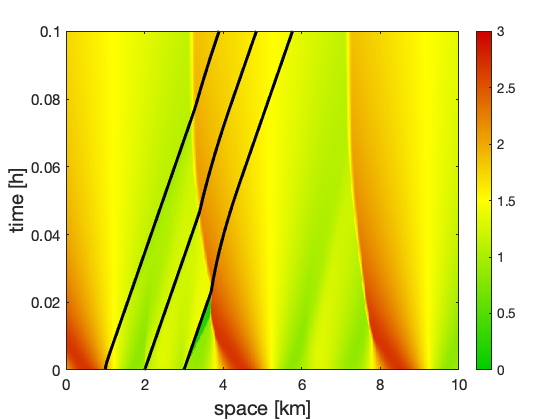} 
 \caption{Total density $r$ and AVs' trajectories}
  \label{fig:totrho} 
\end{subfigure}
\caption{Heat-map of the traffic densities on each lane and the total density on the road segment $[0, 10]$ km and the time interval $[0, 0.1]$ h, with the corresponding AV's trajectories (black lines).}
\label{fig:3lanes}
\end{figure}

In Figure~\ref{fig:3lanes}, we can observe that, due to the different speed limits, vehicles tend to move and accumulate into the fastest lanes 2 and 3, while lane 1 remains almost empty. Also, due to low traffic, on lane 1 there is no interaction between the AV and the bulk traffic. On the contrary, in Figure~\ref{fig:rho3} we can observe  the appearance of non-classical shocks. Moreover, we may observe that AVs slow down to adapt to the downstream traffic speed both in Figures~\ref{fig:rho2} and~\ref{fig:rho3}. Figure~\ref{fig:totrho} gives an overview of the total traffic density on the three lanes, showing that an oscillatory pattern in the density can still be observed due to the uneven distribution of the traffic density across lanes.

\subsection{Relaxation limit $\tau\searrow 0$ without AVs}
\label{sec:relax_num}

To illustrate the convergence result proved in Section~\ref{sec:asymM} (see Theorem~\ref{thm:relax}), we consider the same scenario as in the previous Section~\ref{sec:3lanes}, i.e. $M=3$, $R_j=1$ for $j=1,2,3$ and initial density data as in~\eqref{eq:oscillatingIC}. Now, no AVs are present. \\
We consider two cases:
\begin{itemize}
    \item[{\bf (C1)}] $V_1=V_2=V_3=80$~km/h, so that the speed function is the same on all lanes and we fulfill the hypotheses of Theorem~\ref{thm:relax};
    \item[{\bf (C2)}] $V_1=60$~km/h, $V_2=80$~km/h, $V_3=100$~km/h, to be compared with the solution of the conjectured limit LWR equation~\eqref{eq:limitLWR} with speed flux function $f(r)=r v(r)$, $v(r)=V(1-r/3)$, where $V=80$~km/h is the average maximal speed among lanes 1 to 3. Note that this is indeed the relaxation limit in case {\bf (C1)}.
\end{itemize}
\begin{figure}[ht]
\begin{subfigure}{0.5\textwidth}
  \centering
  \includegraphics[width=\linewidth]{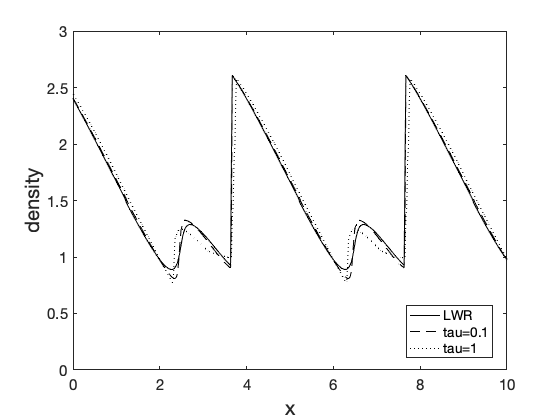} 
  \caption{$V_1=V_2=V_3=80$~km/h}
  \label{fig:relaxHom}
\end{subfigure} 
\begin{subfigure}{0.5\textwidth}
  \centering
  \includegraphics[width=\linewidth]{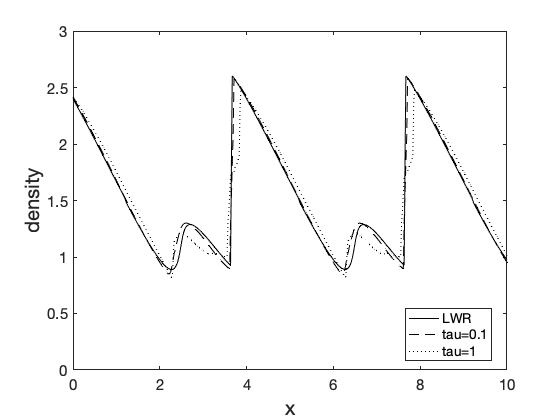} 
 \caption{$V_1=60$, $V_2=80$, $V_3=100$~km/h}
  \label{fig:relaxEtero} 
\end{subfigure} 
\caption{Density profiles at $t=1$~min of the total densities $\rho_1+\rho_2+\rho_3$ of model~\eqref{eq:multiLWR} with $\tau=0.1,1~\hbox{km}^{-1}$ and the solution $r$ of LWR model~\eqref{eq:limitLWR} for cases {\bf (C1)} (left) and {\bf (C2)} (right).}
\label{fig:relax}
\end{figure}

Figure~\ref{fig:relax} shows the profiles of the total density for model~\eqref{eq:multiLWR} for $\tau=1~\hbox{km}^{-1}$ and $\tau=0.1~\hbox{km}^{-1}$ and the solution of the limit LWR problem ~\eqref{eq:limitLWR} at time $t=1$~min. In particular, in Figure~\ref{fig:relaxHom} we can observe an illustration of relaxation limit stated in Theorem~\ref{thm:relax}. Moreover, Figure~\ref{fig:relaxEtero} suggests that the limit also holds for non homogeneous speeds on the different lanes, as conjectured in the design of test {\bf (C2)}.

\subsection{Relaxation limit $\tau\searrow 0$ with AV} \label{sec:relaxAVnum}

Referring to the results in Section~\ref{sec:relaxAV}, we aim at comparing the relaxation limit of model~\eqref{eq:2LWR+AV}, i.e. the multi-lane model with $M=2$ lanes and a single AV, with the solution of the corresponding moving bottleneck model~\eqref{eq:MB2}.

We take $V=V_1=V_2=2$, $R_1=R_2=1$ (thus $R=2$),
so that $F(\rho)=2\rho (1-\rho)$ and $f(r)=2r -r^2$.
Moreover, we set $\tau=0.01$.
We consider different density initial data and AV speeds, to see when non-classical shocks appear in the solution of~\eqref{eq:2LWR+AV} and to understand their nature. In the following, we denote by $\hat r$ and $\check r$ respectively the left and right traces of the non-classical shock in~\eqref{eq:MB2}, 
$\rho^*$ the density
satisfying $F(\rho^*)=u$ and $r^*$ the density
satisfying $f(r^*)=u$.
Setting $u(t)=1$, we get $\check r\simeq 0.15$, $\hat r\simeq 0.85$ $\rho^* =0.5$ and $r^*=1$. Accordingly, we consider the following constant initial data:
\begin{quote}
    \begin{itemize}
        \item[{\bf (IC1)}] $\rho^0_1(x)=\rho^0_2(x)\equiv 0.05$, i.e. $r^0(0)\equiv 0.1$;
        \item[{\bf (IC2)}] $\rho^0_1(x)=\rho^0_2(x)\equiv 0.15$, i.e. $r^0(0)\equiv 0.3$;
        \item[{\bf (IC3)}] $\rho^0_1(x)=\rho^0_2(x)\equiv 0.3$, i.e. $r^0(0)\equiv 0.6$;
        \item[{\bf (IC4)}] $\rho^0_1(x)=\rho^0_2(x)\equiv 0.45$, i.e. $r^0(0)\equiv 0.9$;
        \item[{\bf (IC5)}] $\rho^0_1(x)=\rho^0_2(x)\equiv 0.75$, i.e. $r^0(0)\equiv 1.5$.
    \end{itemize}
\end{quote}
The corresponding density profiles at $t=2$ are displayed in Figure~\ref{fig:2M1AV}. It appears clearly that, when the AV acts as a flux constraint, the density traces at the AV location differ between model~\eqref{eq:2LWR+AV} and~\eqref{eq:MB2}. In particular, we observe that we always have $r(t,y(t)-) > (\rho_1+\rho_2) (t,y(t)-) > (\rho_1+\rho_2) (t,y(t)+) > r(t,y(t)+)$.

\begin{figure}[ht]
\begin{subfigure}{0.5\textwidth}
  \centering
  \includegraphics[width=\linewidth]{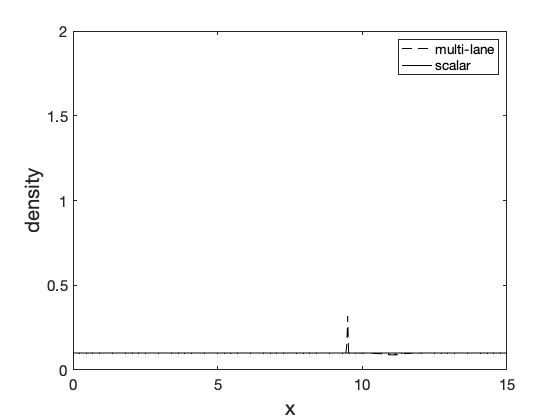} 
  \caption{{\bf (IC1)} $r^0\in\ ]0, \check r[$}
  \label{fig:IC1}
\end{subfigure} 
\begin{subfigure}{0.5\textwidth}
  \centering
  \includegraphics[width=\linewidth]{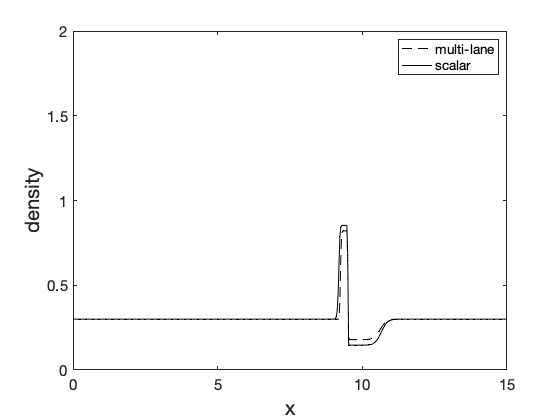} 
  \caption{{\bf (IC2)} $r^0\in\ ]\check r, \rho^*[$}
  \label{fig:IC2}
\end{subfigure} 
\end{figure}
\begin{figure}[ht]\ContinuedFloat
\begin{subfigure}{0.5\textwidth}
  \centering
  \includegraphics[width=\linewidth]{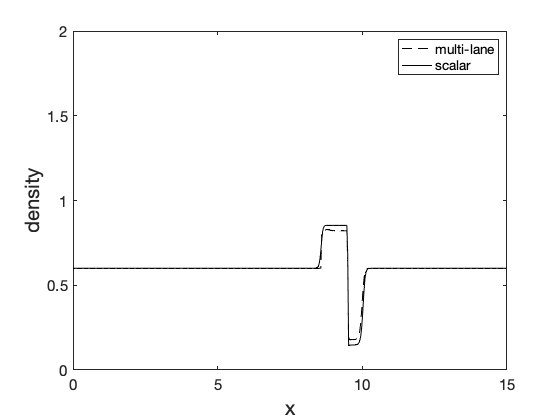} 
  \caption{{\bf (IC3)} $r^0\in\ ]\rho^*,\hat r[$}
  \label{fig:IC3}
\end{subfigure} 
\begin{subfigure}{0.5\textwidth}
  \centering
  \includegraphics[width=\linewidth]{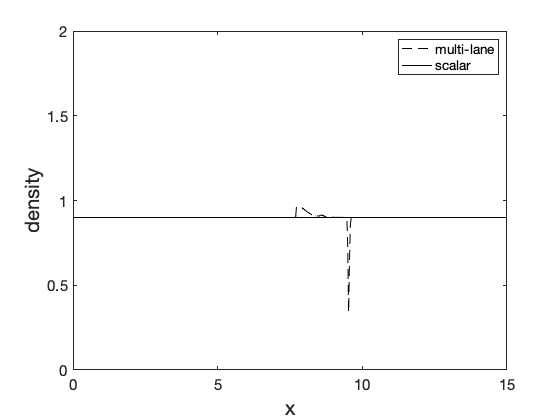} 
  \caption{{\bf (IC4)} $r^0\in\ ]\hat r, r^*[$}
  \label{fig:IC4}
\end{subfigure} 
\end{figure}
\begin{figure}[ht]\ContinuedFloat
\begin{subfigure}{\textwidth}
  \centering
  \includegraphics[width=.5\linewidth]{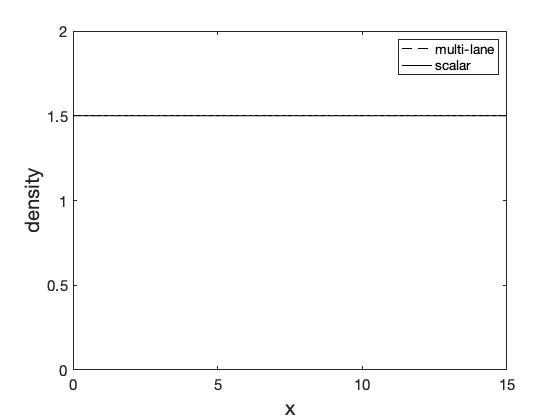} 
  \caption{{\bf (IC5)} $r^0\in\ ]r^*, R[$}
  \label{fig:IC5}
\end{subfigure}
\caption{Comparison between $\rho_1(2, \cdot) + \rho_2 (2,\cdot)$ solutions of~\eqref{eq:2LWR+AV} with $\tau=0.01$ and $r$ solution of~\eqref{eq:MB2},
with $r_0:=\rho_1^0+\rho_2^0$, for $u(t)=1$.}
\label{fig:2M1AV}
\end{figure}

\newpage

\appendix

\section{$\L1$ stability for systems of weakly coupled balance laws}\label{app:stability}

Let us consider the $M\times M$ systems of balance laws coupled in the source term
\begin{subequations}\label{eq:balance}
 \begin{equation} \label{eq:balance1}
\begin{cases}
\del_t \brho + \del_x F(\brho) = G(\brho),\\
\brho(0,x) = \brho^0(x),
\end{cases}
\qquad x\in\R, ~t\geq 0,
\end{equation}
and
 \begin{equation} \label{eq:balance2}
\begin{cases}
\del_t \bsigma + \del_x \tilde F(\bsigma) =  G(\bsigma),\\
\bsigma(0,x) = \bsigma^0(x),
\end{cases}
\qquad x\in\R, ~t\geq 0,
\end{equation}
\end{subequations}
for some $\brho^0,\bsigma^0\in\BV(\R;\R^M)$,
where 
\begin{align*}
    F(\brho)&=\left(F_1(\rho_1),\ldots,F_M(\rho^M) \right)^T, \\
    \tilde F(\brho)&=\left(\tilde F_1(\rho_1),\ldots,\tilde F_M(\rho^M) \right)^T,\\
G(\brho)&=\left(G_1(\brho),\ldots,G_M(\brho)\right)^T.
\end{align*}
Note that system~\eqref{eq:multiLWR} corresponds to setting 
\begin{align*}
    F_j(\rho_j)&=\rho_j v_j(\rho_j), \\
    G_j(\brho)&= \dfrac{1}{\tau}\left(S_{j-1} (\rho_{j-1},\rho_j) - S_{j} (\rho_{j},\rho_{j+1})\right),
\end{align*}
in~\eqref{eq:balance1} for $j=1,\ldots,M$.

The following stability result holds.

\begin{theorem} \label{thm:L1stability}
Let~\eqref{eq:balance1} and~\eqref{eq:balance2} admit unique $\BV$ entropy weak solutions $\brho$, $\bsigma\in\C0\left([0,+\infty[;\L1(\R;\R^M) \right)$ and let $\mathfrak{T}:=\min\left\{\max_{s\in[0,t]}\tv(\brho(s,\cdot),\max_{s\in[0,t]}\tv(\bsigma(s,\cdot))\right\}$
and
$\mathfrak{L}:=\max_{j=1,\ldots,M}\norma{F_j'-\tilde F_j'}_\infty$.
Then for all $t\geq 0$ it holds
\begin{equation}\label{eq:stimaL1bis}
    \norma{\brho(t,\cdot)-\bsigma(t,\cdot)}_1
    \leq 
    e^{\norma{\nabla G}_\infty t} \norma{\brho^0-\bsigma^0}_1
    +\frac{M\, \mathfrak{T}\, \mathfrak{L}}{\norma{\nabla G}_\infty}
    \left( e^{\norma{\nabla G}_\infty t} -1 \right).
\end{equation}
\end{theorem}

\begin{proof}
Let us consider WFT approximate solutions constructed by fractional steps as in Section~\ref{sec:asymM}:
\begin{itemize}
    \item 
In any time interval $[t^n,  t^{n+1}[$, $n\in\N$, $\brho^{\tau,\nu}$ is the WFT approximation of 
\begin{equation} \label{eq:hom_step1app}
\begin{cases}
\del_t \rho_j + \del_x F_j(\rho_j) = 0,  \\
\rho_j(t^{n},x) = \rho_{j}^{\tau,\nu}(t^{n}+,x),
\end{cases}
\qquad 
\begin{array}{l}
x\in\R, ~t\in[t^{n}, t^{n+1}[, \\
j=1,\ldots,M, 
\end{array}
\end{equation}
 and $\bsigma^{\tau,\nu}$ is the WFT approximation of 
\begin{equation} \label{eq:hom_step2app}
\begin{cases}
\del_t \sigma_j + \del_x \tilde F_j(\sigma_j) = 0,  \\
\sigma_j(t^{n},x) = \sigma_{j}^{\tau,\nu}(t^{n}+,x),
\end{cases}
\qquad 
\begin{array}{l}
x\in\R, ~t\in[t^{n}, t^{n+1}[, \\
j=1,\ldots,M, 
\end{array}
\end{equation}
constructed as described e.g. in~\cite[Section 2.3]{HoldenRisebroBook2015}.

\item At time $t=t^{n+1}$, we define
\begin{align} \label{eq:update1app}
\rho_{j}^{\tau,\nu}(t^{n+1}+,\cdot)&=\rho_{j}^{\tau,\nu}(t^{n+1}-,\cdot) 
+ \Delta t^\nu\, G(\brho^{\tau,\nu}(t^{n+1}-,\cdot)), \\
\label{eq:update2app}
\sigma_{j}^{\tau,\nu}(t^{n+1}+,\cdot)&=\sigma_{j}^{\tau,\nu}(t^{n+1}-,\cdot) 
+ \Delta t^\nu\, G(\bsigma^{\tau,\nu}(t^{n+1}-,\cdot)),.
\end{align}
\end{itemize}
Solutions of~\ref{eq:hom_step1app} and~\ref{eq:hom_step2app} satisfy
\begin{equation}\label{eq:FL1app}
    \norma{\rho_{j}^{\tau,\nu}(t^{n+1}-,\cdot) - \sigma_{j}^{\tau,\nu}(t^{n+1}-,\cdot)}_1
    \leq 
    \norma{\rho_{j}^{\tau,\nu}(t^{n}+,\cdot) - \sigma_{j}^{\tau,\nu}(t^{n}+,\cdot)}_1
    +  \mathfrak{T}\, \mathfrak{L} \,  \Delta t^\nu.
\end{equation}
Subtracting~\eqref{eq:update2app} from~\eqref{eq:update1app} and taking $\L1$-norms, we get 
\begin{align*}
    \norma{\rho_{j}^{\tau,\nu}(t^{n+1}+,\cdot) - \sigma_{j}^{\tau,\nu}(t^{n+1}+,\cdot)}_1
    \leq \norma{\rho_{j}^{\tau,\nu}(t^{n+1}-,\cdot)  - \sigma_{j}^{\tau,\nu}(t^{n+1}-,\cdot)}_1
    \left( 1+ \Delta t^\nu \norma{\nabla G}_\infty \right) ,
\end{align*}
which, combined with~\eqref{eq:FL1app}, gives
\begin{align*}
    \sum_{j=1}^M &\, \norma{\rho_{j}^{\tau,\nu}(t^{n}+,\cdot) - \sigma_{j}^{\tau,\nu}(t^{n}+,\cdot)}_1 \\
    &\leq \left(1+ \Delta t^\nu \norma{\nabla G}_\infty \right)^{t^n/\Delta t^\nu} \sum_{j=1}^M\norma{\rho_{j}^0 - \sigma_{j}^0}_1 
     + \Delta t^\nu\,M\, \mathfrak{T}\, \mathfrak{L} \, \sum_{k=1}^{t^n/\Delta t^\nu} \left(1+ \Delta t^\nu \norma{\nabla G}_\infty \right)^k \\
     &= \left(1+ \Delta t^\nu \norma{\nabla G}_\infty \right)^{t^n/\Delta t^\nu} \sum_{j=1}^M\norma{\rho_{j}^0 - \sigma_{j}^0}_1  + \frac{M\, \mathfrak{T}\, \mathfrak{L} }{\norma{\nabla G}_\infty} \left[\left(1+ \Delta t^\nu \norma{\nabla G}_\infty \right)^{t^n/\Delta t^\nu +1}-1\right], 
\end{align*}
which converges to~\eqref{eq:stimaL1bis} as $\Delta t^\nu\searrow 0$.
\end{proof}

\begin{remark}
Note that, setting $G\equiv 0$ in~\eqref{eq:stimaL1bis}, one recovers the classical stability estimate for scalar equations. Instead,~\eqref{eq:L1contraction} is stronger than~\eqref{eq:stimaL1bis}, due to the specific monotonocity assumption {\bf (S0)}. In particular, estimate~\eqref{eq:stimaL1bis} does not suffice to conclude in the proof of Theorem~\ref{thm:relax}, since in this case $\norma{\nabla G}_\infty=2\Sim/\tau$ and we get
\begin{equation*} 
    \sum_{j=1}^M \norma{\rho_j^\tau(t,\cdot) - \frac{r(t,\cdot)}{M}}_{\L1([a,b])}
    = \mathcal{O}(1) \cdot (b-a) \left(\tau e^{2\Sim/\sqrt{\tau}} + e^{-1/\sqrt{\tau}}\right) \quad \xrightarrow[\tau \to 0]{} \quad \infty \,.
\end{equation*}
\end{remark}

{

  \small

 \bibliography{AVmultilane}

\bibliographystyle{abbrv}

}

\end{document}